\documentclass{amsart}
\usepackage{amsfonts}
\usepackage{hyperref}

\newtheorem{theorem}{Theorem}[section]
\theoremstyle{plain}

\newtheorem{conjecture}{Conjecture}
\newtheorem{corollary}{Corollary}[section]

\newtheorem{definition}{Definition}
\newtheorem{example}{Example}

\newtheorem{lemma}{Lemma}

\newtheorem{problem}{Problem}
\newtheorem{proposition}{Proposition}
\newtheorem{remark}{Remark}

\numberwithin{equation}{section}
\def \func {\text}

\begin{document}
\title[Sharp bounds for the Gaussian hypergeometric function]{The first type
of power mean bounds for the Gaussian hypergeometric function}
\author{Zhen-Hang and Yang, Jing-Feng Tian}
\address{Zhen-Hang Yang \\
Engineering Research Center of Intelligent Computing for Complex Energy
Systems of Ministry of Education\\
North China Electric Power University \\
Yonghua Street 619, 071003, Baoding, P. R. China\\
and Zhejiang Society for Electric Power, Hangzhou, Zhejiang, 310014, P. R.
China}
\email{\url{yzhkm@163.com}}
\address{Jing-Feng Tian \\
Department of Mathematics and Physics\\
North China Electric Power University \\
Yonghua Street 619, 071003, Baoding, P. R. China}
\email{\url{tianjf@ncepu.edu.cn}}
\thanks{This paper is in final form and no version of it will be submitted
for publication elsewhere.}

\begin{abstract}
In this paper, by using a signs rule for a class of special power series and
a recurrence relation of coefficients involving hypergeometric series, we
find the necessary and sufficient conditions for the sharp (at the left
endpoint) inequalities between the hypergeometric function and power mean
(Hm-Pm inequality) to hold in all possible cases, which greatly extends
those known results and solve the rest of problem on the sharp (at the left
endpoint) power mean bounds for the hypergeometric mean
\end{abstract}

\maketitle

%\dedicatory{Dedicated to the 60th anniversary of Zhejiang Electric Power Company Research Institute}

\section{Introduction}

For real numbers $a,b$, and $c$ with $c\neq 0,-1,-2,...$, the Gaussian
hypergeometric function is defined by%
\begin{equation*}
F\left( a,b;c;x\right) =\sum_{n=0}^{\infty }\frac{\left( a\right) _{n}\left(
b\right) _{n}}{\left( c\right) _{n}}\frac{x^{n}}{n!}
\end{equation*}%
for $x\in \left( -1,1\right) $, where $\left( a\right) _{n}$ denotes
Pochhammer symbol defined by%
\begin{equation*}
\left( a\right) _{n}=a\left( a+1\right) \cdot \cdot \cdot \left(
a+n-1\right) =\frac{\Gamma \left( n+a\right) }{\Gamma \left( a\right) },
\end{equation*}%
for $n=1,2,...$, and $(a)_{0}=1$ for $a\neq 0$, here $\Gamma
(x)=\int_{0}^{\infty }t^{x-1}e^{-t}dt$ $(x>0)$ is the gamma function.

The hypergeometric function $F(a;b;c;x)$ has the simple differentiation
formulas%
\begin{equation}
F^{\prime }\left( a,b;c;x\right) =\frac{ab}{c}F\left( a+1,b+1;c+1;x\right) .
\label{df}
\end{equation}%
The behavior of the hypergeometric function near $x=1$ in the three cases $%
a+b<c$, $a+b=c$, and $a+b>c$, $a,b,c>0$, is given by%
\begin{equation}
\left\{
\begin{array}{lc}
\bigskip F\left( a,b;c;1\right) =\dfrac{\Gamma \left( c\right) \Gamma \left(
c-a-b\right) }{\Gamma \left( c-a\right) \Gamma \left( c-b\right) } & \text{%
if }c>a+b\text{,} \\
\bigskip F\left( a,b;c;x\right) =\dfrac{R\left( a,b\right) -\ln \left(
1-x\right) }{B\left( a,b\right) }+O\left( \left( 1-x\right) \ln \left(
1-x\right) \right) & \text{if }c=a+b\text{,} \\
F\left( a,b;c;x\right) =(1-x)^{c-a-b}F(c-a,c-b;c;x) & \text{if }c<a+b\text{,}%
\end{array}%
\right.  \label{F-near1}
\end{equation}%
where%
\begin{equation*}
B\left( z,w\right) =\frac{\Gamma \left( z\right) \Gamma \left( w\right) }{%
\Gamma \left( z+w\right) }\text{, \ }\func{Re}\left( z\right) >0\text{, }%
\func{Re}\left( w\right) >0
\end{equation*}%
is the classical beta function,%
\begin{equation}
R\left( a,b\right) =-2\gamma -\psi \left( a\right) -\psi \left( b\right)
\text{,}  \label{R(a,b)}
\end{equation}%
here $\psi \left( z\right) =\Gamma ^{\prime }\left( z\right) /\Gamma \left(
z\right) $, $\func{Re}\left( z\right) >0$ is the psi function and $\gamma $
is the Euler-Mascheroni constant. It should be noted that $F\left(
a,b;a+b;x\right) $ is called zero-balanced function.

As special cases, the complete elliptic integrals of the first and second
kind have the hypergeometric series representations:%
\begin{eqnarray*}
K\left( r\right) &=&\int_{0}^{\pi /2}\frac{1}{\sqrt{1-r^{2}\sin ^{2}t}}dt=%
\frac{\pi }{2}F\left( \frac{1}{2},\frac{1}{2};1;r^{2}\right) , \\
E\left( r\right) &=&\int_{0}^{\pi /2}\sqrt{1-r^{2}\sin ^{2}t}dt=\frac{\pi }{2%
}F\left( -\frac{1}{2},\frac{1}{2};1;r^{2}\right) ,
\end{eqnarray*}%
where $r\in \left( 0,1\right) $. The complete elliptic integrals of the
first kind $K\left( r\right) $ is related to the Gaussian
arithmetic-geometric mean $AGM\left( 1,r^{\prime }\right) $ of $1$ and $%
r^{\prime }=\sqrt{1-r^{2}}\in \left( 0,1\right) $, which is given by Gauss's
formula%
\begin{equation}
AGM\left( 1,r^{\prime }\right) =\frac{\pi }{2K(r)}.  \label{AGM-K}
\end{equation}%
Due to $r^{\prime }<\left( 2/\pi \right) E\left( r\right) <1$, $F\left(
-1/2,1/2;1;r^{2}\right) $ is also a mean of $1$ and $r^{\prime }$.

In 1965, Carlson \cite{Carlson-PAMS-16-1965} defined the hypergeometric mean
value of positive numbers $x_{k}$ ($1\leq k\leq n$) by the R-function (see
\cite{Carlson-JMAA07-1963}). In particular, the Gauss' hypergeometric
function $\left[ F\left( -a,b;c;1-x\right) \right] ^{1/a}$ is a mean of $1$
and $x$ provided $c>b>0$ for $x\in \left( 0,1\right) $. Carlson \cite[%
Theorem 1]{Carlson-PAMS-16-1966} provided power mean bounds for the
hypergeometric mean:
\begin{eqnarray}
\left[ F\left( -a,b;c;1-x\right) \right] ^{1/a} &<&M_{a}\left( x,\frac{b}{c}%
\right) \text{ if }a>1,  \label{CI-a>1} \\
\left[ F\left( -a,b;c;1-x\right) \right] ^{1/a} &>&M_{a}\left( x,\frac{b}{c}%
\right) \text{ if }a<1  \label{CI-a<1}
\end{eqnarray}%
for $x\in \left( 0,1\right) $, where%
\begin{equation*}
M_{p}\left( x,w\right) =\left( 1-w+wx^{p}\right) ^{1/p}\text{ if }p\neq 0%
\text{ and }M_{0}\left( w,x\right) =x^{w}
\end{equation*}%
is the power mean of order $p\in \mathbb{R}$ of $1$ and $x$ with the weight $%
w\in \left( 0,1\right) $. It is known that $p\mapsto M_{p}\left( x,w\right) $
is increasing on $\mathbb{R}$.

Since then, a number of inequalities for certain special hypergeometric
means in terms of various means were found in \cite{Borwein-PA-NY-1974},
\cite{Almkvist-AMM-95-1988}, \cite{Barnard-JMA-31-2000}, \cite%
{Barnard-JMA-32-2000}, \cite{Alzer-JCAM-172-2004}, \cite{Chu-RM-61-2012},
\cite{Yang-JMAA-438-2016}, \cite{Yang-JIA-2017-106}, \cite%
{Yang-JMAA-462-2018}, \cite{Yang-JMAA-467-2018}, \cite{Yang-AMC-348-2019}.
For example, the following sharp inequalities holds for $r\in \left(
0,1\right) $:

\begin{equation*}
M_{0}\left( r^{\prime },\frac{1}{2}\right) <AGM\left( 1,r^{\prime }\right) =%
\frac{\pi }{2K(r)}<M_{1/2}\left( r^{\prime },\frac{1}{2}\right) ,\text{ \
(see \cite{Borwein-PA-NY-1974})}
\end{equation*}%
\begin{equation}
M_{3/2}\left( r^{\prime },\frac{1}{2}\right) <\frac{2}{\pi }E\left( r\right)
<\frac{\pi }{2}M_{q_{0}}\left( r^{\prime },\frac{1}{2}\right) ,\text{ \ (see
\cite{Alzer-JCAM-172-2004}, \cite{Barnard-JMA-31-2000})}  \label{E<A_q0}
\end{equation}%
\begin{equation*}
J_{7/4}\left( 1,r^{\prime }\right) <\frac{2}{\pi }E\left( r\right)
<J_{2/\left( \pi -2\right) }\left( 1,r^{\prime }\right) ,
\end{equation*}%
where $J_{p}\left( 1,r\right) =p\left( r^{p+1}-1\right) /\left[ \left(
p+1\right) \left( r^{p}-1\right) \right] $ is the one-parameter mean, see
\cite{Yang-JMAA-438-2016},

Motivated by Carlson's inequalities (\ref{CI-a>1}) and (\ref{CI-a<1}),
Barnard and Richards \cite[Theorem 1]{Barnard-CMFT-1-2001} showed that, for $%
a\in \left( 0,1\right) $ and $c>b>0$, the inequality%
\begin{equation}
\left[ F\left( -a,b;c;x\right) \right] ^{1/a}>M_{p_{1}}\left( 1-x,\frac{b}{c}%
\right)  \label{HM>Mp1}
\end{equation}%
holds for all $x\in \left( 0,1\right) $, where $p_{1}=\left( a+ab+c-b\right)
/\left( 1+c\right) $. Due to the monotonicity of $p\mapsto M_{p}\left(
x,w\right) $ together with $p_{1}>a$, the inequality (\ref{HM>Mp1}) is a
refinement of (\ref{CI-a<1}). Further, the authors \cite{Barnard-CMFT-1-2001}
conjectured that, for $a\in (0,1)$, $b>0$, and $a+b>1/2$, the sharp
inequality%
\begin{equation*}
M_{p}\left( 1-x,\frac{1}{2}\right) <\left[ F\left( -a,b;2b;x\right) \right]
^{1/a}
\end{equation*}%
holds for all $x\in (0,1)$ if $p<\left( a+2b\right) /\left( 1+2b\right) $.
Also, they proposed the following problem.

\begin{problem}[{\protect\cite[Problem]{Barnard-CMFT-1-2001}}]
\label{P-BR}Given $a>0,c>b>0$, identify conditions on $a,b,c$, and the sharp
value $\varphi \left( a,b,c\right) $ such that if $p>\varphi (a,b,c)$, the
inequality%
\begin{equation*}
\left[ F\left( -a,b;c;x\right) \right] ^{1/a}<M_{p}\left( 1-x,\frac{b}{c}%
\right)
\end{equation*}%
holds for all $x\in \left( 0,1\right) $.
\end{problem}

In 2005, Richards \cite[Theorem 1]{Richards-JMAA-308-2005} confirmed the
conjecture proposed in \cite[Conjecture]{Barnard-CMFT-1-2001}. More
precisely, the author proved a more stronger result than the conjecture.

\begin{theorem}[{\textbf{\protect\cite[Theorem 1]{Richards-JMAA-308-2005}}}]
\textbf{\label{TR}} Suppose $a\leq 1$, $b>0$\ and $c\geq \max \left\{
-a,b\right\} $. If $c\geq \max \left\{ 1-2a,2b\right\} $, then%
\begin{equation}
M_{p}\left( 1-x,\frac{b}{c}\right) \leq \left[ F\left( -a,b;c;x\right) %
\right] ^{1/a},\text{ \ for all }x\in \left( 0,1\right) ,
\label{RI-Mp<Hm-p0}
\end{equation}%
if and only if $p\leq p_{0}$, where%
\begin{equation}
p_{0}=\frac{a+c}{1+c}.  \label{p0}
\end{equation}%
If $c\leq \min \left\{ 1-2a,2b\right\} $, then%
\begin{equation}
\left[ F\left( -a,b;c;x\right) \right] ^{1/a}\leq M_{p}\left( 1-x,\frac{b}{c}%
\right) ,\text{ \ for all }x\in \left( 0,1\right) ,  \label{RI-Hm<Mp-p0}
\end{equation}%
if and only if $p\geq p_{0}$. If $c=1-2a=2b$; or $c=\max \left\{
-a,b\right\} $; or $a=1$, then the equality holds.
\end{theorem}

In the end of \cite{Richards-JMAA-308-2005}, Richards pointed out in \cite[%
Section 4]{Richards-JMAA-308-2005} that, \textquotedblleft Generalizations
of this (that is, the second inequality of (\ref{E<A_q0})) and other sharp
inequalities complementary to Theorem \ref{TR} remain as intriguing problems
for further study\textquotedblright . After two years, Barnard and Richards
\cite{Barnard-JIPAM-8-2007} proved a result related to zero-balanced
hypergeometric functions.

\begin{theorem}[{\protect\cite[Proposition 2.1]{Barnard-JIPAM-8-2007}}]
\label{TBR}Suppose $0<\alpha \leq 1/2$. Then for all $r\in (0,1)$%
\begin{equation*}
M_{\lambda }\left( r^{\alpha },1-\alpha \right) <\frac{1}{F\left( \alpha
,1-\alpha ;1;1-r\right) }<M_{\mu }\left( r^{\alpha },1-\alpha \right)
\end{equation*}%
if and only if $\lambda \leq 0$ and $\mu \geq \left( 1-\alpha \right)
/\left( 2\alpha \right) $.
\end{theorem}

Following the idea of the proof of Theorem \ref{TBR}, Barnard, Richards and
Tiedeman \cite{Barnard-JMI-4-2010} further verified that the function%
\begin{equation*}
r\mapsto M_{0}\left( \left( 1-r\right) ^{-a},\frac{b}{a+b}\right) -F\left(
a,b;a+b;r\right)
\end{equation*}%
is absolutely monotonic on $\left( 0,1\right) $.

In 2010, using Alzer and Qiu's \cite{Alzer-JCAM-172-2004} approach to the
upper bound in the second inequality of (\ref{E<A_q0}) and numerical
evidence, Barnard, Richards and Tiedeman \cite{Barnard-JMI-4-2010}
conjectured the following companion to Theorem \ref{TR}, where%
\begin{equation}
p_{0}^{\ast }=\frac{a\ln \left( 1-b/c\right) }{\ln \frac{\Gamma \left(
c\right) \Gamma \left( c+a-b\right) }{\Gamma \left( c+a\right) \Gamma \left(
c-b\right) }}.  \label{p0*}
\end{equation}

\begin{conjecture}[{\protect\cite[Conjecture I]{Barnard-JMI-4-2010}}]
\label{Cj-BRT}Let $a\leq 1$, $b>0$\ and $c>b>0$. Suppose $c\geq \max \left\{
1-2a,2b\right\} $. Then%
\begin{equation}
\left[ F\left( -a,b;c;x\right) \right] ^{1/a}\leq M_{p}\left( 1-x,\frac{b}{c}%
\right) \text{ \ for all }x\in \left( 0,1\right) ,  \label{H<M-BRT1}
\end{equation}%
if $p\geq p_{0}^{\ast }$ (sharp). Suppose $c\leq \min \left\{
1-2a,2b\right\} $. Then the inequality (\ref{H<M-BRT1}) reverses $p\leq
p_{0}^{\ast }$ (sharp).
\end{conjecture}

Applying Pales' Comparison Theorem \cite{Pales-JMAA-131-1988}, the authors
verified the above conjecture is valid for certain special cases.

\begin{theorem}[{\protect\cite[Proposition 1]{Barnard-JMI-4-2010}}]
\label{T-BRT}Suppose $a\in (-1,1)$. If $a>-1/2$, then for all $r\in (0,1)$,%
\begin{equation*}
M_{\lambda }\left( 1-r,\frac{1}{2}\right) <\left[ F\left( -a,1;2;r\right) %
\right] ^{1/a}<M_{\mu }\left( 1-r,\frac{1}{2}\right)
\end{equation*}%
if $\lambda \leq \left( a+2\right) /3$ (sharp) and $\mu \geq \left( a\ln
2\right) /\ln \left( 1+a\right) $ (sharp). If $a<-1/2$, then the above
double inequality holds for all $r\in (0,1)$ if $\lambda \leq \left( a\ln
2\right) /\ln \left( 1+a\right) $ (sharp) and $\mu \geq \left( a+2\right) /3$
(sharp).
\end{theorem}

\begin{remark}
Clearly, Theorem \ref{TR} provides a type of sharp power mean bounds $%
M_{p_{0}}\left( 1-x,b/c\right) $ for the hypergeometric mean $\left[ F\left(
-a,b;c;x\right) \right] ^{1/a}$, which is sharp at $x=0$. If Conjecture \ref%
{Cj-BRT} is proved to be true, then it provides another type of sharp power
mean bounds $M_{p_{0}^{\ast }}\left( 1-x,b/c\right) $ for the hypergeometric
mean $\left[ F\left( -a,b;c;x\right) \right] ^{1/a}$, which is sharp at $x=1$%
. For convenience, we call them \textquotedblleft the first and second type
of power mean bounds\textquotedblright , respectively. Moreover, we call the
inequity%
\begin{equation}
\left[ F\left( -a,b;c;x\right) \right] ^{1/a}\leq \left( \geq \right)
M_{p}\left( 1-x,\frac{b}{c}\right) \text{ for }x\in \left( 0,1\right)
\label{Hm<>Pm}
\end{equation}%
as "Hm-Pm inequality". If such inequality is sharp at $x=0$ ($x=1$) then it
is called \textquotedblleft the first (second) Hm-Pm
inequality\textquotedblright .
\end{remark}

From Theorem \ref{TR}, a problem naturally rises.

\begin{problem}
\label{P-CTR}Suppose $a\leq 1$, $b>0$\ and $c\geq \max \left\{ -a,b\right\} $%
. When $\min \left\{ 1-2a,2b\right\} <c<\max \left\{ 1-2a,2b\right\} $, what
are the conditions such that the hypergeometric mean $\left[ F\left(
-a,b;c;x\right) \right] ^{1/a}$ has the first type of power mean bounds $%
M_{p_{0}}\left( 1-x,b/c\right) $?
\end{problem}

If Problem \ref{P-CTR} is solved, then the first Hm-Pm inequality in Problem %
\ref{P-BR} is completely tackled. The aim of this paper is to weaken those
conditions in Theorem \ref{TR} for the first Hm-Pm inequality to hold.
Precisely, we shall find necessary and sufficient conditions for the first
Hm-Pm inequalities to hold in all possible cases, which contains Theorem \ref%
{TR} and gives an answer to Problem \ref{P-CTR}.

As shown in \cite[p. 304]{Richards-JMAA-308-2005}, if $a=1$; or $c=\max
\{-a,b\}$; or $c=1-2a=2b$, then%
\begin{equation}
\left[ F\left( -a,b;c;x\right) \right] ^{1/a}\equiv M_{p_{0}}\left( 1-x,%
\frac{b}{c}\right) ,  \label{Fm=Mp0}
\end{equation}%
where $p_{0}$ is defined by (\ref{p0}). In what follows, unless otherwise
noted, we always assume that

(i) $\left( a,b,c\right) \in \left\{ \left( a,b,c\right) :a<1,b>0,c>\max
\left\{ -a,b\right\} \right\} :=B\left( a,b,c\right) $. Then $p_{0}\in
\left( 0,1\right) $.

(ii) As shown in Remark \ref{R-gp=0}, if $\left( a,b,c\right) \in B\left(
a,b,c\right) $ then the identity (\ref{Fm=Mp0}) holds for $x\in \left(
0,1\right) $ if and only if $c=1-2a=2b$. This fact will be no longer listed
in our results.

(iii) If $a=0$, then $\left[ F\left( -a,b;c;x\right) \right] ^{1/a}$ and $%
p_{0}^{\ast }$ are regarded as their limits at $a=0$, that are,%
\begin{eqnarray*}
\lim_{a\rightarrow 0}\left[ F\left( -a,b;c;x\right) \right] ^{1/a} &=&\exp %
\left[ -\sum_{n=1}^{\infty }\frac{\left( b\right) _{n}}{\left( c\right) _{n}}%
\frac{x^{n}}{n}\right] , \\
\lim_{a\rightarrow 0}p_{0}^{\ast } &=&\frac{\ln \left( 1-b/c\right) }{\psi
\left( c-b\right) -\psi \left( c\right) }.
\end{eqnarray*}

The paper is organized as follows. Our main results are stated in Section 2,
where our first main result, Theorem \ref{MT-m}, is listed in Section 2.1,
completely describes the monotonicity pattern of $f_{p_{0}}\left( x\right) $
on $\left( 0,1\right) $, our third and fourth results, Theorems \ref{MT-1}
and \ref{MT-2} presented in Section 2.2, completely offer necessary and
sufficient conditions for the first Hm-Pm inequality, and several examples,
as applications of our main results, are listed Section 2.4. In Section 3,
we introduce two important tools, in which the \textquotedblleft
Negative-Positive type power series\textquotedblright\ and their signs rule
are arranged in Subsection 3.1, a recurrence relation involving the
hypergeometric function is established in Subsection 3.2. In Section 4, we
consider the derivative of the function $f_{p}\left( x\right) $ defined on $%
\left( 0,1\right) $ by
\begin{equation}
f_{p}\left( x\right) =\frac{\left[ F\left( -a,b;c;x\right) \right] ^{1/a}}{%
\left[ 1-b/c+\left( b/c\right) \left( 1-x\right) ^{p}\right] ^{1/p}},
\label{fp}
\end{equation}%
and obtain $\func{sgn}f_{p}^{\prime }\left( x\right) =\func{sgn}g_{p}\left(
x\right) $ and $g_{p}\left( x\right) =\sum_{n=0}^{\infty }\xi _{n}x^{n}$,
where $g_{p}\left( x\right) $ is given by (\ref{gp}). In Subsection 4.1, we
deduce the recurrence relations of $\xi _{n}$ in two forms; in Subsections
4.2 and 4.3, we determine the signs of $\gamma _{n}\left( p_{0}\right) $ and
$g_{p}\left( 1\right) $. These statements, as lemmas, are directly used in
the proofs of main results. The proofs of main results are presented in
Sections 5 and 6.

\bigskip

\section{Main results and examples}

Before stating main results, we introduce several notations, that are,%
\begin{eqnarray}
\delta \left( p\right) &=&a-b+c-p,  \label{dp} \\
C_{p} &=&f_{p}\left( 1\right) =\left( 1-\frac{b}{c}\right) ^{-1/p}\left[
\frac{\Gamma \left( c\right) \Gamma \left( c+a-b\right) }{\Gamma \left(
c+a\right) \Gamma \left( c-b\right) }\right] ^{1/a},  \label{Cp}
\end{eqnarray}%
if $b>0$ and $c>\max \left\{ b,b-a\right\} $.

\begin{remark}
\label{R-Cp0-1}We claim that%
\begin{equation}
\left[ \frac{\Gamma \left( c\right) \Gamma \left( c+a-b\right) }{\Gamma
\left( c+a\right) \Gamma \left( c-b\right) }\right] ^{1/a}<1\text{ \ or \ }%
\frac{1}{a}\ln \frac{\Gamma \left( c\right) \Gamma \left( c+a-b\right) }{%
\Gamma \left( c+a\right) \Gamma \left( c-b\right) }<0  \label{/p0*<0}
\end{equation}%
if $a\neq 0$, $b>0$ and $c>\max \left\{ b,b-a\right\} $. In fact, in view of
$\left( \ln \Gamma \left( x\right) \right) ^{\prime \prime }=\psi ^{\prime
}\left( x\right) >0$ for $x>0$, by the property of convex function it is
derived that%
\begin{eqnarray*}
\frac{1}{a}\ln \frac{\Gamma \left( c\right) \Gamma \left( c+a-b\right) }{%
\Gamma \left( c+a\right) \Gamma \left( c-b\right) } &=&\frac{\ln \Gamma
\left( c-b+a\right) -\ln \Gamma \left( c-b\right) }{a} \\
&&-\frac{\ln \Gamma \left( c+a\right) -\ln \Gamma \left( c\right) }{a}<0.
\end{eqnarray*}%
This implies that $C_{p_{0}}\geq \left( \leq \right) 1$ is equivalent to $%
p_{0}\leq \left( \geq \right) p_{0}^{\ast }$, which gives%
\begin{equation}
\min \left\{ C_{p_{0}},1\right\} =1\text{ iff }p_{0}\leq p_{0}^{\ast }\text{
\ and \ }\max \left\{ C_{p_{0}},1\right\} =1\text{ iff }p_{0}\geq
p_{0}^{\ast }\text{.}  \label{Cp0,1-p0,p0*}
\end{equation}
\end{remark}

\subsection{Monotonicity of $f_{p_{0}}\left( x\right) $ and $f_{p}\left(
x\right) $}

Our first main result completely describes the monotonicity pattern of $%
f_{p_{0}}\left( x\right) $ on $\left( 0,1\right) $, which records as follows.

\begin{theorem}
\label{MT-m}Suppose that $a<1$, $b>0$ and $c>\max \left\{ -a,b\right\} $.
Let $p_{0}$ and $C_{p_{0}}$ be defined by (\ref{p0}) and (\ref{Cp}),
respectively.

(i) The function $f_{p_{0}}$ defined by (\ref{fp}) is increasing on $\left(
0,1\right) $ if and only if $c\geq 2b$ and $a-b+c\geq p_{0}$. Therefore, the
double inequality%
\begin{equation}
M_{p_{0}}\left( 1-x,\frac{b}{c}\right) <\left[ F\left( -a,b;c;x\right) %
\right] ^{1/a}<C_{p_{0}}M_{p_{0}}\left( 1-x,\frac{b}{c}\right)
\label{Pm<Hm<CPm-p0}
\end{equation}%
holds for $x\in \left( 0,1\right) $ if $c\geq 2b$ and $a-b+c\geq p_{0}$.

(ii) If $c\geq 2b$ and $a-b+c<p_{0}$, then there is an $x_{0}\in \left(
0,1\right) $ such that the function $f_{p_{0}}$ is increasing on $\left(
0,x_{0}\right) $ and decreasing on $\left( x_{0},1\right) $. Consequently,
the inequality%
\begin{equation}
\min \left\{ 1,C_{p_{0}}\right\} M_{p_{0}}\left( 1-x,\frac{b}{c}\right) <%
\left[ F\left( -a,b;c;x\right) \right] ^{1/a}  \label{Pm<Hm-p0}
\end{equation}%
holds for $x\in \left( 0,1\right) $ if $a-b+c>0$.

(iii) The function $f_{p_{0}}$ is decreasing on $\left( 0,1\right) $ if and
only if $c\leq 2b$ and $a-b+c\leq p_{0}$. Therefore, if $c\leq 2b$ and $%
a-b+c\leq p_{0}$, then the inequality%
\begin{equation}
\left[ F\left( -a,b;c;x\right) \right] ^{1/a}<M_{p_{0}}\left( 1-x,\frac{b}{c}%
\right)  \label{Hm<Pm-p0}
\end{equation}%
holds for $x\in \left( 0,1\right) $. Further, the double inequality%
\begin{equation}
C_{p_{0}}M_{p_{0}}\left( 1-x,\frac{b}{c}\right) <\left[ F\left(
-a,b;c;x\right) \right] ^{1/a}<M_{p_{0}}\left( 1-x,\frac{b}{c}\right)
\label{CPm<Hm<Pm-p0}
\end{equation}%
holds for $x\in \left( 0,1\right) $ if $a-b+c>0$.

(iv) If $c\leq 2b$ and $a-b+c>p_{0}$, then there is an $x_{0}\in \left(
0,1\right) $ such that the function $f_{p_{0}}$ is decreasing on $\left(
0,x_{0}\right) $ and increasing on $\left( x_{0},1\right) $. Consequently,
the inequality%
\begin{equation}
\left[ F\left( -a,b;c;x\right) \right] ^{1/a}<\max \left\{
1,C_{p_{0}}\right\} M_{p_{0}}\left( 1-x,\frac{b}{c}\right)  \label{Hm<cPm-p0}
\end{equation}%
holds for $x\in \left( 0,1\right) $.
\end{theorem}

\begin{remark}
As a by-product, Theorem \ref{MT-m} implies an inequality between $p_{0}$
and $p_{0}^{\ast }$. In fact, from the inequalities connecting the first and
third functions in (\ref{Pm<Hm<CPm-p0}) and (\ref{Pm<Hm<CPm-p0}) it follows
that%
\begin{equation*}
C_{p_{0}}>\left( <\right) 1\text{ if }c\geq \left( <\right) 2b\text{ and\ }%
0<a-b+c\geq \left( \leq \right) p_{0},
\end{equation*}%
which, by (\ref{Cp0,1-p0,p0*}), leads to%
\begin{equation}
p_{0}\leq \left( \geq \right) p_{0}^{\ast }\text{ if }c\geq \left( <\right)
2b\text{ \ and \ }0<a-b+c\geq \left( \leq \right) p_{0}.  \label{p0<>p*}
\end{equation}
\end{remark}

Our second main result provides the sufficient conditions for the
monotonicity of $f_{p}\left( x\right) $ for $p\in \mathbb{R}$, which is
contained in the following theorem.

\begin{theorem}
\label{T-fp-m}Suppose that $a<1$, $b>0$ and $c>\max \left\{ -a,b\right\} $.
Let $p_{0}$ be defined by (\ref{p0}) and $p\in \mathbb{R}$.

(i) If $c\geq 2b$ and $p\leq p_{0}\leq a-b+c$, then the function $f_{p}$
defined by (\ref{fp}) is increasing on $\left( 0,1\right) $, and therefore,
the double inequality%
\begin{equation}
M_{p}\left( 1-x,\frac{b}{c}\right) <\left[ F\left( -a,b;c;x\right) \right]
^{1/a}<C_{p}M_{p}\left( 1-x,\frac{b}{c}\right)  \label{Pm<Hm<CPm-p}
\end{equation}%
holds for $x\in \left( 0,1\right) $, where $C_{p}$ is defined by (\ref{Cp}).

(ii) If $c\leq 2b$ and $p\geq p_{0}\geq a-b+c$, then the function $f_{p}$ is
decreasing on $\left( 0,1\right) $, and therefore, the reverse of the left
hand side inequality of (\ref{Pm<Hm<CPm-p}) holds for $x\in \left(
0,1\right) $. Further, if $a-b+c>0$, then the double inequality (\ref%
{Pm<Hm<CPm-p}) reverses.
\end{theorem}

\begin{proof}
We first show that the function%
\begin{equation*}
x\mapsto \frac{f_{p}\left( x\right) }{f_{p_{0}}\left( x\right) }=\frac{\left[
1-b/c+\left( b/c\right) \left( 1-x\right) ^{p_{0}}\right] ^{1/p_{0}}}{\left[
1-b/c+\left( b/c\right) \left( 1-x\right) ^{p}\right] ^{1/p}}
\end{equation*}%
is increasing (decreasing) on $\left( 0,1\right) $ if $p\leq \left( \geq
\right) p_{0}$. In fact, logarithmic differentiation yields%
\begin{eqnarray*}
\left[ \ln \frac{f_{p}\left( x\right) }{f_{p_{0}}\left( x\right) }\right]
^{\prime } &=&b\frac{\left( 1-x\right) ^{p-1}}{c-b+b\left( 1-x\right) ^{p}}-b%
\frac{\left( 1-x\right) ^{p_{0}-1}}{c-b+b\left( 1-x\right) ^{p_{0}}} \\
&=&b\left( c-b\right) \frac{t^{p-1}-t^{p_{0}-1}}{\left( c-b+bt^{p}\right)
\left( c-b+bt^{p_{0}}\right) }\geq \left( \leq \right) 0
\end{eqnarray*}%
for $t=1-x\in \left( 0,1\right) $ if $p\leq \left( \geq \right) p_{0}$. Now
using Theorem \ref{MT-m} (i) we deduce that, the function%
\begin{equation*}
x\mapsto f_{p}\left( x\right) =\frac{f_{p}\left( x\right) }{f_{p_{0}}\left(
x\right) }\times f_{p_{0}}\left( x\right)
\end{equation*}%
is increasing on $\left( 0,1\right) $ if $c\geq 2b$, $a-b+c\geq p_{0}$ and $%
p\leq p_{0}$. Similarly, $f_{p}\left( x\right) $ is decreasing on $\left(
0,1\right) $ if $c\leq 2b$, $a-b+c\leq p_{0}$ and $p\geq p_{0}$. The
required inequalities follow from the monotonicity of $f_{p}\left( x\right) $%
, which completes the proof.
\end{proof}

\begin{remark}
The double inequality \ref{Pm<Hm<CPm-p} and its reverse provide a maximum
relative error of estimate of $\left[ F\left( -a,b;c;x\right) \right] ^{1/a}$
by $M_{p}\left( 1-x,b/c\right) $:%
\begin{equation*}
\left\vert \frac{\left[ F\left( -a,b;c;x\right) \right] ^{1/a}-M_{p}\left(
1-x,b/c\right) }{\left[ F\left( -a,b;c;x\right) \right] ^{1/a}}\right\vert
<\left\vert \frac{C_{p}-1}{C_{p}}\right\vert .
\end{equation*}
\end{remark}

\begin{remark}
We claim that $p\mapsto C_{p}M_{p}\left( 1-x,b/c\right) $ is decreasing on $%
\left( 0,\infty \right) $. In fact, it suffices to prove that $\phi ^{\prime
}\left( p\right) <0$ for $p>0$, where%
\begin{equation*}
\phi \left( p\right) =\ln \left[ \left( \frac{c}{c-b}\right)
^{1/p}M_{p}\left( 1-x,\frac{b}{c}\right) \right] =\frac{1}{p}\ln \left(
1+\sigma t^{p}\right)
\end{equation*}%
with $\sigma =b/\left( c-b\right) $ and $t=1-x\in \left( 0,1\right) $.
Differentiation leads to%
\begin{eqnarray*}
p^{2}\phi ^{\prime }\left( p\right) &=&\frac{\sigma pt^{p}}{1+\sigma t^{p}}%
\ln t-\ln \left( 1+\sigma t^{p}\right) :=\phi _{1}\left( p\right) , \\
\phi _{1}^{\prime }\left( p\right) &=&p\sigma \frac{t^{p}\ln ^{2}t}{\left(
1+\sigma t^{p}\right) ^{2}}>0
\end{eqnarray*}%
for $p>0$. This gives $\phi _{1}\left( p\right) <\lim_{p\rightarrow \infty
}\phi _{1}\left( p\right) =0$, so $\phi ^{\prime }\left( p\right) <0$ for $%
p>0$, thereby completing the proof.
\end{remark}

\subsection{The first Hm-Pm inequalities for $p=p_{0}$}

Our third and fourth main results are contained in Theorems \ref{MT-1} and %
\ref{MT-2}, which completely offer necessary and sufficient conditions for
the first Hm-Pm inequality.

\begin{theorem}
\label{MT-1}Suppose that $a<1$, $b>0$ and $c>\max \left\{ -a,b\right\} $.
Let $p_{0}$ and $p_{0}^{\ast }$ be defined by (\ref{p0}) and (\ref{p0*}),
respectively.

(i) If $a-b+c\geq p_{0}$, then the first Hm-Pm inequality%
\begin{equation}
\left[ F\left( -a,b;c;x\right) \right] ^{1/a}>M_{p_{0}}\left( 1-x,\frac{b}{c}%
\right)   \label{Hm>Pm-p0}
\end{equation}%
holds for $x\in \left( 0,1\right) $ if and only if $c\geq 2b$.

(ii) If $a-b+c<p_{0}$, then the first Hm-Pm inequality (\ref{Hm>Pm-p0})
holds for $x\in \left( 0,1\right) $ if and only if $a-b+c>0$, $c\geq 2b$ and
$p_{0}\leq p_{0}^{\ast }$.
\end{theorem}

\begin{theorem}
\label{MT-2}Suppose that $a<1$, $b>0$ and $c>\max \left\{ -a,b\right\} $.
Let $p_{0}$ and $p_{0}^{\ast }$ be defined by (\ref{p0}) and (\ref{p0*}),
respectively.

(i) If $a-b+c\leq p_{0}$, then the first Hm-Pm inequality (\ref{Hm<Pm-p0})
holds for $x\in \left( 0,1\right) $ if and only if $c\leq 2b$.

(ii) If $a-b+c>p_{0}$, then the first Hm-Pm inequality (\ref{Hm<Pm-p0})
holds for $x\in \left( 0,1\right) $ if and only if $c\leq 2b$ and $p_{0}\geq
p_{0}^{\ast }$.
\end{theorem}

By the inequality (\ref{p0<>p*}) we have%
\begin{eqnarray*}
\left\{ a-b+c\geq p_{0},c\geq 2b\right\} &=&\left\{ a-b+c\geq p_{0},c\geq
2b,p_{0}\leq p_{0}^{\ast }\right\} , \\
\left\{ 0<a-b+c\leq p_{0},c\leq 2b\right\} &=&\left\{ 0<a-b+c\leq
p_{0},c\leq 2b,p_{0}\geq p_{0}^{\ast }\right\} .
\end{eqnarray*}%
Using the above relations and replacing the assumption that
\textquotedblleft $c>\max \left\{ -a,b\right\} $\textquotedblright\ by
\textquotedblleft $c>\max \left\{ b,b-a\right\} $\textquotedblright\ in
Theorems \ref{MT-1} and \ref{MT-2}, we can obtain the the following nice
conclusion.

\begin{corollary}
\label{C-MT1+2}Suppose that $a<1$, $b>0$ and $c>\max \left\{ b,b-a\right\} $%
. Let $p_{0}$ and $p_{0}^{\ast }$ be defined by (\ref{p0}) and (\ref{p0*}),
respectively.

(i) The first Hm-Pm inequality (\ref{Hm>Pm}) holds for $x\in \left(
0,1\right) $ if and only if $c\geq 2b$ and $p_{0}\leq p_{0}^{\ast }$.

(ii) The first Hm-Pm inequality (\ref{Hm<Pm}) inequality holds for $x\in
\left( 0,1\right) $ if and only if $c\leq 2b$ and $p_{0}\geq p_{0}^{\ast }$.
\end{corollary}

\subsection{The first Hm-Pm inequalities for $p\leq \left( \geq \right) p_{0}
$ }

Using Theorems \ref{MT-1} and \ref{MT-2} we can deduce our fifth main
result, which shows the necessary and sufficient condition for the general
Hm-Pm inequalities.

\begin{theorem}
\label{MT-3}Suppose that $a<1$, $b>0$ and $c>\max \left\{ -a,b\right\} $.
Let $p_{0}$ and $p_{0}^{\ast }$ be defined by (\ref{p0}) and (\ref{p0*}),
respectively.

(i) If $c\geq 2b$ and $a-b+c\geq p_{0}$, or $c\geq 2b$ and $%
0<a-b+c<p_{0}\leq p_{0}^{\ast }$, then the first Hm-Pm inequality%
\begin{equation}
\left[ F\left( -a,b;c;x\right) \right] ^{1/a}>M_{p}\left( 1-x,\frac{b}{c}%
\right)   \label{Hm>Pm}
\end{equation}%
holds for $x\in \left( 0,1\right) $ if and only if $p\leq p_{0}$.

(ii) $c\leq 2b$ and $a-b+c\leq p_{0}$, or $c\leq 2b$ and $a-b+c>p_{0}\geq
p_{0}^{\ast }$, then the first Hm-Pm inequality%
\begin{equation}
\left[ F\left( -a,b;c;x\right) \right] ^{1/a}<M_{p}\left( 1-x,\frac{b}{c}%
\right)   \label{Hm<Pm}
\end{equation}%
holds for $x\in \left( 0,1\right) $ if and only if $p\geq p_{0}$.
\end{theorem}

\begin{proof}
The necessary condition for the first Hm-Pm inequality (\ref{Hm>Pm}) ((\ref%
{Hm<Pm})) to holds for $x\in \left( 0,1\right) $ follows from the limit
relation%
\begin{eqnarray*}
&&\lim\limits_{x\rightarrow 0}\dfrac{a^{-1}\ln F\left( -a,b;c;x\right) -\ln
M_{p}\left( 1-x,b/c\right) }{x^{2}} \\
&=&-\dfrac{b\left( c-b\right) }{2c^{2}}\left( p-\dfrac{a+c}{1+c}\right) \geq
\left( \leq \right) 0.
\end{eqnarray*}%
The sufficiency follows from Theorems \ref{MT-1} and \ref{MT-2}, thereby
completing the proof.
\end{proof}

\begin{remark}
\label{R-dp0-c-2b-1-2a}Note that $\delta \left( p_{0}\right) =a-b+c-p_{0}$
can be written as
\begin{equation}
\delta \left( p_{0}\right) =\frac{a+c}{c+1}\left( c-2b\right) +\frac{b}{c+1}%
\left[ c-\left( 1-2a\right) \right] .  \label{dp0-c-2b-1-2a}
\end{equation}
Then $\delta \left( p_{0}\right) \geq 0$ if $c\geq \max \left\{
2b,1-2a\right\} $ and $\delta \left( p_{0}\right) \leq 0$ if $c\leq \min
\left\{ 2b,1-2a\right\} $. By Theorem \ref{MT-3}, Richards' Theorem \ref{TR}
follows immediately. It thus can be seen that Richards' Theorem \ref{TR} is
a corollary of our main results.
\end{remark}

The next corollary gives an answer to Problem \ref{P-CTR}, which is a
perfect complement to Ricards' Theorem \ref{TR}.

\begin{corollary}
\label{C-c-2b,1-2a}Suppose that $a<1$, $b>0$, $c>\max \left\{ -a,b\right\} $
and $\min \left\{ 2b,1-2a\right\} <c<\max \left\{ 2b,1-2a\right\} $ Let $%
p_{0}$ and $p_{0}^{\ast }$ be defined by (\ref{p0}) and (\ref{p0*}),
respectively, and%
\begin{equation*}
c_{0}=\frac{1}{2}\sqrt{\left( b-a\right) ^{2}+4b}+\frac{1}{2}\left(
b-a\right) .
\end{equation*}%
(i) If $a+b<1/2$ and $c_{0}\leq c<1-2a$, or $a+b<1/2$, $2b<c<c_{0}$ and $%
p_{0}\leq p_{0}^{\ast }$ if $a-b+c>0$, then the first Hm-Pm inequality (\ref%
{Hm>Pm}) holds if and only if $p\leq p_{0}$.

(ii) If $a+b>1/2$ and $1-2a<c\leq c_{0}$, or $a+b>1/2$, $c_{0}\leq c\leq 2b$
and $p_{0}\geq p_{0}^{\ast }$ if $a-b+c>0$, then the  first Hm-Pm inequality
(\ref{Hm<Pm}) holds if and only if $p\geq p_{0}$.
\end{corollary}

\begin{proof}
We denote%
\begin{equation*}
\rho \left( c\right) =\delta \left( p_{0}\right) =a-b+c-\frac{a+c}{1+c}=%
\frac{\left( c-b+c+a\right) ^{2}-\left[ \left( b-a\right) ^{2}+4b\right] }{%
4\left( c+1\right) }.
\end{equation*}%
Since $b>0$ and $c>\max \left\{ -a,b\right\} $, $\rho \left( c\right) $ has
a unique zero $c_{0}$. Differentiation yields%
\begin{equation*}
\rho ^{\prime }\left( c\right) =\frac{\left( c^{2}+c\right) +\left(
c+a\right) }{\left( c+1\right) ^{2}}>0
\end{equation*}%
with%
\begin{eqnarray*}
\rho \left( 2b\right) &=&\frac{b}{2b+1}\left( 2a+2b-1\right) , \\
\rho \left( 1-2a\right) &=&-\frac{1}{2}\left( 2a+2b-1\right) .
\end{eqnarray*}%
Case 1: $a+b<1/2$. Then $2b<c<1-2a$. Since $\rho \left( 2b\right) <0$ and $%
\rho \left( 1-2a\right) >0$, we see that $\rho \left( c\right) =\delta
\left( p_{0}\right) <0$ for $c\in \left( 2b,c_{0}\right) $ and $\rho \left(
c\right) =\delta \left( p_{0}\right) >0$ for $c\in \left( c_{0},1-2a\right) $%
.

Case 2: $a+b>1/2$. Then $1-2a<c<2b$. Similarly, it is seen that $\rho \left(
c\right) =\delta \left( p_{0}\right) <0$ for $c\in \left( 1-2a,c_{0}\right) $
and $\rho \left( c\right) =\delta \left( p_{0}\right) >0$ for $c\in \left(
c_{0},2b\right) $.

Using Theorem \ref{MT-3} the required results follow, which completes the
proof.
\end{proof}

Letting $c=b-a$. That $b-a>\max \left\{ -a,b\right\} $ implies $a<0$ and
that $c\leq 2b$ implies $b\geq -a>0$. Replacing $-a$ by $a$ in Theorem \ref%
{MT-2} we obtain a sharp upper bound for the zero-balanced hypergeometric
mean in term of power mean.

\begin{corollary}
\label{C2-T2}Let $b\geq a>0$. Then the inequality%
\begin{equation}
\left[ F\left( a,b;b+a;x\right) \right] ^{-1/a}<M_{p}\left( 1-x,\dfrac{b}{a+b%
}\right)  \label{ZHm<Mp}
\end{equation}%
holds for $x\in \left( 0,1\right) $ if and only if $p\geq b/\left(
1+a+b\right) $. In particular, we have%
\begin{equation}
\left[ F\left( b,b;2b;x\right) \right] ^{-1/b}<M_{b/\left( 1+2b\right)
}\left( 1-x,\dfrac{1}{2}\right)  \label{ZHm<Mp-a=b}
\end{equation}%
for $x\in \left( 0,1\right) $ with the best constant $b/\left( 1+2b\right) $.
\end{corollary}

\begin{remark}
The inequality (\ref{ZHm<Mp}) is clearly a generalization of Barnard and
Richards' result in \cite[Proposition 2.1]{Barnard-JIPAM-8-2007}.
\end{remark}

\subsection{Applications}

We list several examples to illustrate applications of our main results.

\begin{example}
\label{E1-c=2b}Suppose $a<1$. If $a\geq -1/2$, then for all $x\in (0,1)$,%
\begin{equation}
\left[ F\left( -a,1;2;x\right) \right] ^{1/a}\geq M_{p}\left( 1-x,\frac{1}{2}%
\right)  \label{Hm<Pm--a,1,2}
\end{equation}%
holds if and only if $p\leq \left( a+2\right) /3$. If $a\leq -1/2$, then the
inequality (\ref{Hm<Pm--a,1,2}) reveses if and only if $p\geq \left(
a+2\right) /3$.
\end{example}

\begin{proof}
A simple computation yields that $c=2b$, $p_{0}=\left( a+2\right) /3$, $%
\delta \left( p_{0}\right) =a-b+c-p_{0}=\left( 2a+1\right) /3$ and%
\begin{equation*}
p_{0}^{\ast }=\frac{a\ln 2}{\ln \left( 1+a\right) }\text{ if }a\neq 0\text{
\ and \ }p_{0}^{\ast }=\ln 2\text{ \ for }a\in \left( -1,1\right) \text{.}
\end{equation*}%
We first show that%
\begin{equation}
p_{0}<p_{0}^{\ast }\text{ for }a\in \left( -1/2,1\right) \text{ \ and \ }%
p_{0}>p_{0}^{\ast }\text{ for }a\in \left( -1,-1/2\right) .  \label{p0<>p0*}
\end{equation}%
In fact, since $p_{0}^{\ast }$ is clearly positive for $a>-1$, it suffices
to prove that%
\begin{equation*}
\frac{p_{0}}{p_{0}^{\ast }}=\frac{1}{3\ln 2}\frac{\left( a+2\right) \ln
\left( a+1\right) }{a}\text{, }
\end{equation*}%
is increasing for $a\in \left( 0,\infty \right) $ and decreasing for $a\in
\left( -1,0\right) $. Differentiation yields%
\begin{eqnarray*}
\left( 3a^{2}\ln 2\right) \frac{d}{da}\frac{p_{0}}{p_{0}^{\ast }} &=&\frac{%
a\left( a+2\right) }{\left( a+1\right) }-2\ln \left( a+1\right) =\mu \left(
a\right) , \\
\mu ^{\prime }\left( a\right) &=&\frac{a^{2}}{\left( a+1\right) ^{2}}\geq 0%
\text{.}
\end{eqnarray*}%
This together with the facts that $\left. \left( p_{0}/p_{0}^{\ast }\right)
\right\vert _{a=1}=\left. \left( p_{0}/p_{0}^{\ast }\right) \right\vert
_{a=-1/2}=1$ leads to the inequalities (\ref{p0<>p0*}).

(i) If $a\geq -1/2$, then $\delta \left( p_{0}\right) =a-b+c-p_{0}\geq 0$
and $c=2b$, but the condition that \textquotedblleft $0<a-b+c<p_{0}\leq
p_{0}^{\ast }$ and $c\geq 2b$\textquotedblright\ is not satisfied. using
Theorem \ref{MT-3} (i), the inequality (\ref{Hm<Pm--a,1,2}) holds for $x\in
\left( 0,1\right) $ if and only if $p\leq p_{0}=\left( a+2\right) /3$

(ii) If $a\leq -1/2$, then $\delta \left( p_{0}\right) =a-b+c-p_{0}\leq 0$
and $c=2b$, but the condition that \textquotedblleft $a-b+c>p_{0}\geq
p_{0}^{\ast }$ and $c\leq 2b$\textquotedblright\ is not satisfied. using
Theorem \ref{MT-3} (ii), the inequality (\ref{Hm<Pm--a,1,2}) reverses if and
only if $p\geq p_{0}=\left( a+2\right) /3$

This completes the proof.
\end{proof}

\begin{remark}
\label{R-E1-fr}By Theorem \ref{T-fp-m}, if $a\geq \left( \leq \right) -1/2$
then the function%
\begin{equation*}
x\mapsto \frac{\left[ F\left( -a,1;2;x\right) \right] ^{1/a}}{M_{p}\left(
1-x,1/2\right) }
\end{equation*}%
is increasing (decreasing) on $\left( 0,1\right) $. Further, if $a\in \left(
-1,1\right) $, we have
\begin{equation}
M_{p}\left( 1-x,\frac{1}{2}\right) <\left[ F\left( -a,1;2;x\right) \right]
^{1/a}<C_{p}M_{p}\left( 1-x,\frac{1}{2}\right)  \label{Hm-Pm-Cp}
\end{equation}%
for $x\in \left( 0,1\right) $, where $C_{p}=2^{1/p}\left( a+1\right) ^{-1/a}$
if $a\neq 0$ and $C_{p}=2^{1/p}/e$ if $a=0$.
\end{remark}

\begin{remark}
Putting $a=-1$ in Example \ref{E1-c=2b} gives%
\begin{equation*}
\left[ F\left( 1,1;2;x\right) \right] ^{-1}<M_{1/3}\left( 1-x,\frac{1}{2}%
\right) .
\end{equation*}%
Note that%
\begin{equation*}
\left[ F\left( 1,1;2;1-x\right) \right] ^{-1}=\frac{x-1}{\ln x}=L\left(
x,1\right)
\end{equation*}%
is the logarithmic mean of $1$ and $x\in \left( 0,1\right) $. So this
inequality can be written as $L\left( x,1\right) <M_{1/3}\left( x,1/2\right)
$ for $x\in \left( 0,1\right) $, which is known as Lin's inequality \cite%
{Ling-AMM-81-1974} (see also \cite{Yang-JIA-2008-149286}, \cite%
{Yang-JIA-2013-116}).
\end{remark}

\begin{remark}
Letting $a=0$ in Example \ref{E1-c=2b} gives%
\begin{equation}
\lim_{a\rightarrow 0}\left[ F\left( -a,1;2;x\right) \right]
^{1/a}>M_{2/3}\left( 1-x,\frac{1}{2}\right)  \label{Hm>Mp0-a=0}
\end{equation}%
for $x\in \left( 0,1\right) $ with the best constant $p_{0}=2/3$. Note that%
\begin{eqnarray*}
\lim_{a\rightarrow 0}\left[ F\left( -a,1;2;x\right) \right] ^{1/a} &=&\exp %
\left[ -\sum_{n=1}^{\infty }\frac{x^{n}}{n\left( n+1\right) }\right] \\
&=&\exp \left[ -\frac{\left( 1-x\right) \ln \left( 1-x\right) }{x}-1\right]
=I\left( 1,1-x\right) ,
\end{eqnarray*}%
where $I\left( 1,1-x\right) $ is the exponential mean of $1$ and $\left(
1-x\right) \in \left( 0,1\right) $. Then the inequality (\ref{Hm>Mp0-a=0})
can be written as%
\begin{equation*}
I\left( 1,x\right) >M_{2/3}\left( x,\frac{1}{2}\right)
\end{equation*}%
for $x\in \left( 0,1\right) $, which was proven in \cite%
{Pittinger-UBPEFSMF-680-1980}. (see also \cite{Yang-MIA-10-2007}). Further,
by (\ref{Hm-Pm-Cp}) we have%
\begin{equation*}
M_{2/3}\left( x,\frac{1}{2}\right) <I\left( 1,x\right) <C_{2/3}M_{2/3}\left(
x,\frac{1}{2}\right)
\end{equation*}%
for $x\in \left( 0,1\right) $, where $C_{2/3}=2\sqrt{2}e^{-1}=1.040...$ is
the best constant (see \cite{Yang-MIA-10-2007}).
\end{remark}

\begin{example}
Taking $\left( a,b,c\right) =\left( 1/2,1/2,1\right) $ yields%
\begin{equation*}
c-2b=0\text{ \ and \ }a-b+c=1>p_{0}=\frac{3}{4}.
\end{equation*}%
Using Theorem \ref{MT-1} (i) yields%
\begin{equation*}
\left[ F\left( -\frac{1}{2},\frac{1}{2};1;x\right) \right]
^{2}>M_{3/4}\left( 1-x,\frac{1}{2}\right)
\end{equation*}%
for $x\in \left( 0,1\right) $ with the best constants $p_{0}=3/4$. This
inequality was proven in \cite{Barnard-JMA-31-2000}. Moreover, by Theorem %
\ref{MT-m} (i) we also have%
\begin{equation*}
M_{3/4}\left( 1-x,\frac{1}{2}\right) <\left[ F\left( -\frac{1}{2},\frac{1}{2}%
;1;x\right) \right] ^{2}<C_{3/4}M_{3/4}\left( 1-x,\frac{1}{2}\right)
\end{equation*}%
for $x\in \left( 0,1\right) $, where $C_{3/4}=8\sqrt[3]{2}/\pi ^{2}=1.021...$
is the best constant.
\end{example}

\begin{example}
Taking $\left( a,b\right) =\left( 1/2,1/2\right) $ in Corollary \ref{C2-T2}
gives%
\begin{equation*}
F\left( \frac{1}{2},\frac{1}{2};1;x\right) ^{-2}<M_{1/4}\left( 1-x,\frac{1}{2%
}\right)
\end{equation*}%
for $x\in \left( 0,1\right) $ with the best constant $1/4$. This inequality
is equivalent to%
\begin{equation*}
AGM\left( x,1\right) <M_{1/2}\left( x,\frac{1}{2}\right)
\end{equation*}%
for $x\in \left( 0,1\right) $ (see \cite{Borwein-PA-NY-1974}).
\end{example}

\begin{example}
Taking $\left( a,b,c\right) =\left( -3/2,1/2,3\right) $ leads to%
\begin{equation*}
p_{0}=\frac{3}{8}\text{, \ }c-2b=2>0\text{ \ and \ }a-b+c-p_{0}=\frac{5}{8}>0
\end{equation*}%
Using Theorem \ref{MT-1} (i) gives%
\begin{equation*}
\left[ F\left( \frac{3}{2},\frac{1}{2};3;x\right) \right] ^{-2/3}>M_{3/8}%
\left( 1-x,\frac{1}{6}\right)
\end{equation*}%
for $x\in \left( 0,1\right) $ with the best constants $p_{0}=3/8$.
\end{example}

\begin{example}
Taking $\left( a,b,c\right) =\left( -1/2,9/4,3\right) $. Then%
\begin{equation*}
p_{0}=\frac{5}{8}\text{, \ }c-2b=-\frac{3}{2}<0\text{ \ and \ }a-b+c-p_{0}=-%
\frac{3}{8}<0
\end{equation*}%
By Theorem \ref{MT-2} (ii) we obtain that%
\begin{equation*}
\left[ F\left( \frac{1}{2},\frac{9}{4};3;x\right) \right] ^{-2}<M_{p}\left(
1-x,\frac{3}{4}\right)
\end{equation*}%
for $x\in \left( 0,1\right) $ with the best constants $p_{0}=5/8$.
\end{example}

\begin{remark}
In the above two examples, $\min \left\{ 2b,1-2a\right\} <c<\max \left\{
2b,1-2a\right\} $. Richards' Theorem \ref{TR} does not work.
\end{remark}

The following example is an application of Corollary \ref{C-c-2b,1-2a} in
the case of $\min \left\{ 2b,1-2a\right\} <c<\max \left\{ 2b,1-2a\right\} $.

\begin{example}
Taking $\left( a,b,c\right) =\left( -9/8,2/3,2\right) $. Then $b>0$, $%
c=2>9/8=\max \left\{ b,-a\right\} $ and%
\begin{equation*}
\frac{4}{3}=2b<c=2<1-2a=\frac{13}{4}.
\end{equation*}%
An easy computation leads to $a-b+c=5/24>0$, $p_{0}=7/24=0.291...$ and%
\begin{equation*}
p_{0}^{\ast }=-\frac{9}{8}\frac{\ln \left( 2/3\right) }{\ln \frac{\Gamma
\left( 2\right) \Gamma \left( 5/24\right) }{\Gamma \left( 4/3\right) \Gamma
\left( 7/8\right) }}=0.302...,
\end{equation*}%
then $p_{0}<p_{0}^{\ast }$. Also,%
\begin{equation*}
c_{0}=\frac{1}{2}\sqrt{\left( b-a\right) ^{2}+4b}+\frac{1}{2}\left(
b-a\right) =\frac{43+\sqrt{3385}}{48}=2.107...\text{.}
\end{equation*}%
By Corollary \ref{C-c-2b,1-2a} (i) we see that the inequality%
\begin{equation*}
\left[ F\left( \frac{9}{8},\frac{2}{3};2;x\right) \right] ^{-8/9}>M_{p}%
\left( 1-x;\frac{1}{3}\right)
\end{equation*}%
holds for $x\in \left( 0,1\right) $ if and only if $p\leq 7/24$.
\end{example}

\begin{remark}
In fact, if we apply Theorem \ref{MT-3} (i), then the required inequality is
immediate since it is easy to verify that $c\geq 2b$ and $%
0<a-b+c<p_{0}<p_{0}^{\ast }$.
\end{remark}

\section{Two tools}

To prove our main results, we need two tools. The first tool is the NP and
PN-type power series and their signs rule. Recall that if every term of a
real sequence is nonnegative (nonpositive) and at least one is non-zero,
then this sequence is called a \textbf{positive (negative) sequence}. Then
the so-called the NP and PN-type sequences can be defined as follows.

\begin{definition}
\label{D-NPs}Let $m\in \mathbb{N}$. A real sequence $\left\{ a_{n}\right\}
_{n\geq 0}$ is called an negative-positive-type sequence, NP-type sequence
for short, if the subsequences $\left\{ a_{n}\right\} _{0\leq n\leq m}$ and $%
\left\{ a_{n}\right\} _{n>m}$ are negative and positive sequences,
respectively. $\left\{ -a_{n}\right\} _{n\geq 0}$ is called a
positive-negative-type sequence, PN-type sequence for short.
\end{definition}

Based on Definition \ref{D-NPs}, the NP and PN-type power series defined in
\cite{Yang-RACSAM-115-2021-8} can be redefined as follows.

\begin{definition}
\label{D-NPp}If a real sequence $\left\{ a_{n}\right\} _{n\geq 0}$ is
NP-type, then the normal power series $S\left( t\right) =\sum_{k=0}^{\infty
}a_{k}t^{k}$ is called an NP-type power series, and $-S\left( t\right) $ is
called a PN-type power series.
\end{definition}

For the NP and PN-type power series, a simple but efficient criterion to
determine their signs has been proven in \cite{Yang-JMAA-462-2018}, \cite%
{Yang-AADM-13-2019}, which is a revised version of the electronic preprint
\cite{Yang-arXiv-1705.05703}.

\begin{proposition}
\label{P-sgnS}Let $S\left( t\right) $ be an NP-type power series converging
on the interval $\left( 0,r\right) $ ($r>0$). (i) If $S\left( r^{-}\right)
\leq 0$, then $S\left( t\right) <0$ for all $t\in \left( 0,r\right) $. (ii)
If $S\left( r^{-}\right) >0$, then there is a unique $t_{0}\in \left(
0,r\right) $ such that $S\left( t\right) <0$ for $t\in \left( 0,t_{0}\right)
$ and $S\left( t\right) >0$ for $t\in \left( t_{0},r\right) $.
\end{proposition}

\begin{remark}
If $r=\infty $, then Proposition \ref{P-sgnS} is reduced to \cite[Lemma 6.3]%
{Belzunce-I:M-E-40-2007} (see also \cite[Lemma 2]{Yang-JIA-157-2015}, \cite[%
Lemma 2.1]{Yang-JIA-299-2015}).
\end{remark}

\begin{remark}
\label{R-S><0}For an NP-type power series, it follows from Proposition \ref%
{P-sgnS} that if there is a $t_{1}\in \left( 0,r\right) $ such that $S\left(
t_{1}\right) <0$, then we have $S\left( t\right) <0$ for all $t\in \left(
0,t_{1}\right) $; if there is a $t_{2}\in \left( 0,r\right) $ such that $%
S\left( t_{2}\right) >0$, then we have $S\left( t\right) >0$ for all $t\in
\left( t_{2},r\right) $. For a PN-type power series $-S\left( t\right) $,
the signs of $-S\left( t\right) $ on $\left( 0,t_{1}\right) $ and $\left(
t_{2},r\right) $ are reversed.
\end{remark}

\begin{remark}
In Definition \ref{D-NPs}, if there is a $k_{0}\in \mathbb{N}$ such that $%
a_{m+k_{0}}\neq 0$ and $a_{k}=0$ for $k>m+k_{0}$, then $S\left( t\right)
=\sum_{k=0}^{\infty }a_{k}t^{k}$ is in fact a polynomial of degree $m+k_{0}$%
, and we call it as NP-type polynomial. And, Proposition \ref{P-sgnS} yields
\cite[Lemma 7]{Yang-AAA-2014-702718}, which offers a signs rule of NP and
PN-type polynomials, see also \cite{Yang-JIA-2017-210},\ \cite%
{Yang-JMI-12-2018}.
\end{remark}

Proposition \ref{P-sgnS} has been applied in the study for certain special
functions, see for example, \cite{Belzunce-I:M-E-40-2007}, \cite%
{Yang-JIA-157-2015}, \cite{Yang-MIA-21-2018-2}, \cite{Yang-MIA-21-2018-4},
\cite{Yang-JMAA-467-2018}, \cite{Yang-AMC-348-2019}, \cite%
{Yang-JMAA-470-2019}.

The second tool is a recurrence relation involving the hypergeometric
functions, which first appeared in \cite[Corollary 2]{Yang-arxiv-2204.04709}%
. Several special cases of Proposition \ref{P-PH-rr} can be seen in \cite%
{Yang-JMAA-467-2018}, \cite{Yang-RACSAM-115-2021}, \cite%
{Zhao-RACSAM-117-2023}.

\begin{proposition}[{\protect\cite[Corollary 2]{Yang-arxiv-2204.04709}}]
\label{P-PH-rr}Let $\left( a-1\right) \left( c-a\right) \neq 0$ and $%
-c\notin \mathbb{N\cup }\left\{ 0\right\} $. Then for $p\in \mathbb{R}$, we
have%
\begin{equation}
U\left( x\right) =\left( 1-x\right) ^{p}F\left( a,b;c;x\right)
=\sum_{n=0}^{\infty }u_{n}x^{n},  \label{U}
\end{equation}%
with $u_{0}=1$, $u_{1}=ab/c-p$ and for $n\geq 1$,%
\begin{eqnarray*}
u_{n+1} &=&\frac{2n^{2}+\left( a+b+c-2p-1\right) n+ab-cp}{\left( n+1\right)
\left( n+c\right) }u_{n} \\
&&-\frac{\left( n+a-p-1\right) \left( n+b-p-1\right) }{\left( n+1\right)
\left( n+c\right) }u_{n-1}.
\end{eqnarray*}
\end{proposition}

\section{Lemmas}

Logarithmic differentiation for (\ref{fp}) yields%
\begin{eqnarray}
\frac{f_{p}^{\prime }\left( x\right) }{f_{p}\left( x\right) } &=&\frac{1}{a}%
\frac{\left( -ab/c\right) F\left( 1-a,b+1;c+1;x\right) }{F\left(
-a,b;c;x\right) }+\frac{1}{p}\frac{p\left( b/c\right) \left( 1-x\right)
^{p-1}}{\left( 1-b/c+\left( b/c\right) \left( 1-x\right) ^{p}\right) }
\notag \\
&=&\frac{b\left( c-b\right) }{c}\frac{g_{p}\left( x\right) }{\left(
c-b+b\left( 1-x\right) ^{p}\right) F\left( -a,b;c;x\right) },  \label{dfp}
\end{eqnarray}%
where%
\begin{equation}
\begin{array}{rl}
g_{p}\left( x\right) & =\dfrac{c}{c-b}\left( 1-x\right) ^{p-1}F\left(
-a,b;c;x\right) \bigskip \\
& -\left( 1+\dfrac{b}{c-b}\left( 1-x\right) ^{p}\right) F\left(
1-a,b+1;c+1;x\right) .%
\end{array}
\label{gp}
\end{equation}%
By Gauss contiguous relations \cite[p. 558]{Abramowitz-HMFFGMT-1972}, $%
g_{p}\left( x\right) $ can be simplified as%
\begin{equation}
g_{p}\left( x\right) =\left( 1-x\right) ^{p-1}F\left( 1-a,b;c+1;x\right)
-F\left( 1-a,b+1;c+1;x\right) \text{.}  \label{gp-a1}
\end{equation}%
Employing Cauchy product formula, $g_{p}\left( x\right) $ has the following
power series representation:%
\begin{equation*}
g_{p}\left( x\right) =\sum_{n=0}^{\infty }\left( u_{n}-v_{n}\right)
x^{n}=\sum_{n=0}^{\infty }\xi _{n}x^{n},
\end{equation*}%
where $\xi _{n}\equiv \xi _{n}\left( p\right) =u_{n}-v_{n}$ with%
\begin{eqnarray}
u_{n} &\equiv &u_{n}\left( p\right) =\sum_{k=0}^{n}\frac{\left( 1-p\right)
_{n-k}}{\left( n-k\right) !}\frac{\left( 1-a\right) _{k}\left( b\right) _{k}%
}{k!\left( c+1\right) _{k}},  \label{un} \\
v_{n} &=&\frac{\left( 1-a\right) _{n}\left( b+1\right) _{n}}{n!\left(
c+1\right) _{n}}  \label{vn}
\end{eqnarray}%
for all $n\geq 0$.

\begin{remark}
Evidently, $u_{n},v_{n}>0$ for all $n\geq 0$ if $1>a$, $c>b>0$ and $p\in
\left( 0,1\right) $.
\end{remark}

In this section, we first establish the recurrence relations of $\xi
_{n}=\xi _{n}\left( p\right) $, then deal with the signs of $\gamma
_{n}=\gamma _{n}\left( p_{0}\right) $, and finally determine the signs of $%
g_{p}\left( 1\right) $ and $g_{p_{0}}\left( 1\right) $.

\subsection{Recurrence relations of the coefficients in the power series of $%
g_{p}\left( x\right) $}

In order to confirm the changes of sign of $\xi _{n}$, we need to establish
the recurrence relation of the sequence $\left\{ \xi _{n}\right\} _{n\geq 0}$%
.

\begin{lemma}
\label{L-xin-rr}Let $u_{n}$ and $v_{n}$ be defined by (\ref{un}) and (\ref%
{vn}), respectively. The sequence $\xi _{n}\equiv \xi _{n}\left( p\right)
=u_{n}-v_{n}$ satisfies the recurrence relations%
\begin{equation}
\xi _{n+1}-\alpha _{n}\xi _{n}+\beta _{n}\xi _{n-1}=\tau _{n}\gamma
_{n}v_{n-1}  \label{xin-rr-0}
\end{equation}%
where\
\begin{eqnarray}
\alpha _{n} &=&\frac{2n^{2}+\left( c-a+b-2p+3\right) n+\left( 1-a\right)
b+\left( c+1\right) \left( 1-p\right) }{\left( n+1\right) \left(
n+c+1\right) },  \label{alphn} \\
\beta _{n} &=&\frac{\left( n+1-a-p\right) \left( n+b-p\right) }{\left(
n+1\right) \left( n+c+1\right) },  \label{betan}
\end{eqnarray}%
\begin{eqnarray}
\gamma _{n} &=&\gamma _{n}\left( p\right) =n^{2}p\left( a-b+c-p\right)
+ab\left( p\left( c+1\right) -\left( c+a\right) \right)  \label{gamn} \\
&&-n\left[ cp^{2}-\left( a-b+c+2ab\right) p+a\left( a+c\right) \right] ,
\notag \\
\tau _{n} &=&\left[ n\left( n+1\right) \left( n+c\right) \left( n+c+1\right) %
\right] ^{-1}.  \label{tn}
\end{eqnarray}%
Moreover, if $a<1,c>b>0$, $p\in \left( 0,1\right) $ then $\alpha _{n},\beta
_{n}>0$ for $n\geq 1$.
\end{lemma}

\begin{proof}
Since%
\begin{equation*}
\left( 1-x\right) ^{p-1}F\left( 1-a,b;c+1;x\right) =\sum_{n=0}^{\infty
}u_{n}x^{n},
\end{equation*}%
where $u_{n}$ is defined by (\ref{un}), replacing $\left( p,a,b,c\right) $
by $\left( p-1,1-a,b,c+1\right) $ in Proposition \ref{P-PH-rr} yields%
\begin{eqnarray*}
u_{n+1} &=&\frac{2n^{2}+\left( c-a+b-2p+3\right) n+\left( 1-a\right)
b+\left( c+1\right) \left( 1-p\right) }{\left( n+1\right) \left(
n+c+1\right) }u_{n} \\
&&-\frac{\left( n+1-a-p\right) \left( n+b-p\right) }{\left( n+1\right)
\left( n+c+1\right) }u_{n-1}=\alpha _{n}u_{n}-\beta _{n}u_{n-1}.
\end{eqnarray*}%
Then $\xi _{n}=u_{n}-v_{n}$ satisfies the recurrence relation%
\begin{equation*}
\xi _{n+1}+v_{n+1}=\alpha _{n}\left( \xi _{n}+v_{n}\right) -\beta _{n}\left(
\xi _{n-1}+v_{n-1}\right) ,
\end{equation*}%
which can be written as%
\begin{eqnarray}
\xi _{n+1}-\alpha _{n}\xi _{n}+\beta _{n}\xi _{n-1} &=&-v_{n+1}+\alpha
_{n}v_{n}-\beta _{n}v_{n-1}  \label{xin-rr-1} \\
&=&\left( -\frac{v_{n+1}}{v_{n-1}}+\alpha _{n}\frac{v_{n}}{v_{n-1}}-\beta
_{n}\right) v_{n-1}\text{.}  \notag
\end{eqnarray}%
It is easy to verify that%
\begin{equation}
\frac{v_{n+1}}{v_{n}}=\frac{\left( n+1-a\right) \left( n+b+1\right) }{\left(
n+1\right) \left( n+c+1\right) },  \label{vn+1/vn}
\end{equation}%
from which it is deduced that%
\begin{equation}
-\frac{v_{n+1}}{v_{n-1}}+\alpha _{n}\frac{v_{n}}{v_{n-1}}-\beta _{n}=\frac{%
\gamma _{n}}{n\left( n+1\right) \left( n+c\right) \left( n+c+1\right) }.
\label{vn-gn}
\end{equation}%
Thus the recurrence relation (\ref{xin-rr-0}) follows.

Finally, the positivity of $\alpha _{n}$ and $\beta _{n}$ is obvious,
thereby completing the proof.
\end{proof}

\begin{lemma}
\label{L-xin-rr-a}Let $u_{n}$ and $v_{n}$ be defined by (\ref{un}) and (\ref%
{vn}), respectively. The sequence $\xi _{n}\equiv \xi _{n}\left( p\right)
=u_{n}-v_{n}$ satisfies the recurrence relation%
\begin{equation}
\left( \xi _{n+1}-\theta _{n}\xi _{n}\right) -\lambda _{n}\left( \xi
_{n}-\theta _{n-1}\xi _{n-1}\right) =\tau _{n}\gamma _{n}u_{n-1},
\label{xin-rr-2}
\end{equation}%
where%
\begin{eqnarray}
\theta _{n} &=&\frac{v_{n+1}}{v_{n}}=\frac{\left( n+1-a\right) \left(
n+b+1\right) }{\left( n+1\right) \left( n+c+1\right) },  \label{thn} \\
\lambda _{n} &=&\lambda _{n}\left( p\right) =\alpha _{n}-\theta _{n}=\frac{%
n^{2}+\left( c-2p+1\right) n+\left( a+c-p-cp\right) }{\left( n+1\right)
\left( n+c+1\right) },  \label{ln}
\end{eqnarray}%
$\gamma _{n}=\gamma _{n}\left( p\right) $ and $\tau _{n}$ are given by (\ref%
{gamn}) and (\ref{tn}), respectively. Moreover, if $a<1$, $b>0$, $c>\max
\left\{ -a,b\right\} $, then $\theta _{n}>0$ for $n\geq 1$, $\lambda
_{n}\left( p\right) >0$ for $n\geq 2$ and $\lambda _{1}\left( p_{0}\right)
>0 $.
\end{lemma}

\begin{proof}
We write the relation (\ref{xin-rr-1}) as
\begin{eqnarray*}
&&\left( \xi _{n+1}-\frac{v_{n+1}}{v_{n}}\xi _{n}\right) -\left( \alpha _{n}-%
\frac{v_{n+1}}{v_{n}}\right) \left( \xi _{n}-\frac{v_{n}}{v_{n-1}}\xi
_{n-1}\right) \\
&=&\left( \left( \alpha _{n}-\frac{v_{n+1}}{v_{n}}\right) \frac{v_{n}}{%
v_{n-1}}-\beta _{n}\right) \xi _{n-1}-v_{n+1}+\alpha _{n}v_{n}-\beta
_{n}v_{n-1}
\end{eqnarray*}%
Substituting $\xi _{n-1}=u_{n-1}-v_{n-1}$ into the right hand side of the
above equation and simplifying yield%
\begin{eqnarray*}
&&\left( \xi _{n+1}-\frac{v_{n+1}}{v_{n}}\xi _{n}\right) -\left( \alpha _{n}-%
\frac{v_{n+1}}{v_{n}}\right) \left( \xi _{n}-\frac{v_{n}}{v_{n-1}}\xi
_{n-1}\right) \\
&=&\left( -\frac{v_{n+1}}{v_{n-1}}+\alpha _{n}\frac{v_{n}}{v_{n-1}}-\beta
_{n}\right) u_{n-1},
\end{eqnarray*}%
which, by (\ref{vn-gn}), implies (\ref{xin-rr-2}). The formula (\ref{thn})
follows from (\ref{vn+1/vn}), and (\ref{ln}) follows by an elementary
computation.

It is clear that $\theta _{n}>0$ for $n\geq 1$ provided $a<1,b,c>0$. The
numerator of the fraction in (\ref{ln}) can be written as%
\begin{equation*}
\left( n-2\right) n+2\left( 1-p\right) n+\left( a+c\right) +\left(
n-p\right) \left( c+1\right) ,
\end{equation*}%
which is positive for $n\geq 2$ if $a<1$, $b>0$, $c>\max \left\{
-a,b\right\} $. As for $\lambda _{1}\left( p_{0}\right) >0$, which follows
from%
\begin{equation*}
\lambda _{1}\left( p_{0}\right) =\frac{c^{2}+c+2\left( 1-a\right) }{2\left(
c+1\right) \left( c+2\right) }>0,
\end{equation*}%
thereby completing the proof.
\end{proof}

\subsection{The signs of $\protect\gamma _{n}\left( p_{0}\right) $}

In Lemmas \ref{L-xin-rr} and \ref{L-xin-rr-a}, those coefficients $\alpha
_{n}$, $\beta _{n}$, $\theta _{n}$, and $\lambda _{n}$ are positive under
the basic assumptions, except $\gamma _{n}$. The following lemma gives the
conditions for $\gamma _{n}\left( p_{0}\right) \geq \left( \leq \right) 0$.

\begin{lemma}
\label{L-sgngn-p0}Suppose that $a<1$, $b>0$, $c>\max \left\{ -a,b\right\} $.
Let $\gamma _{n}\equiv \gamma _{n}\left( p\right) $ and $\delta \left(
p\right) $ be defined by (\ref{gamn}) and (\ref{dp}), respectively.

(i) If $c\geq \left( \leq \right) 2b$ and $\delta \left( p_{0}\right) \geq
\left( \leq \right) 0$, then $\gamma _{n}\left( p_{0}\right) \geq \left(
\leq \right) 0$ for all $n\geq 1$.

(ii) If $c>\left( <\right) 2b$ and $\delta \left( p_{0}\right) <\left(
>\right) 0$, then there is an $n_{1}>1$ such that $\gamma _{n}\left(
p_{0}\right) \geq \left( \leq \right) 0$ for $1\leq n\leq n_{1}$ and $\gamma
_{n}\left( p_{0}\right) <\left( >\right) 0$ for $n\geq n_{1}+1$.
\end{lemma}

\begin{proof}
Since $p_{0}=\left( c+a\right) /\left( c+1\right) $, we have%
\begin{eqnarray*}
\frac{1}{n}\gamma _{n}\left( p_{0}\right) &=&np_{0}\left( a-b+c-p_{0}\right)
\\
&&-cp_{0}^{2}+\left( a-b+c+2ab\right) p_{0}-a\left( a+c\right) .
\end{eqnarray*}%
A direct verification yields%
\begin{equation*}
\gamma _{1}\left( p_{0}\right) =\frac{\left( 1-a\right) \left( a+c\right) }{%
c+1}\left( c-2b\right) .
\end{equation*}

Then%
\begin{eqnarray*}
\frac{1}{n}\gamma _{n}\left( p_{0}\right) &=&\left[ \frac{1}{n}\gamma
_{n}\left( p_{0}\right) -\gamma _{1}\left( p_{0}\right) \right] +\gamma
_{1}\left( p_{0}\right) \\
&=&\left( n-1\right) p_{0}\delta \left( p_{0}\right) +\frac{\left(
1-a\right) \left( a+c\right) }{c+1}\left( c-2b\right) .
\end{eqnarray*}

This shows that $\gamma _{n}\left( p_{0}\right) \geq \left( \leq \right) 0$
for all $n\geq 1$ if $c\geq \left( \leq \right) 2b$ and $\delta \left(
p_{0}\right) \geq \left( \leq \right) 0$.

In the case that $c>\left( <\right) 2b$ and $\delta \left( p_{0}\right)
<\left( >\right) 0$, since $\delta \left( p_{0}\right) <\left( >\right) 0$,
the sequence $\left\{ n^{-1}\gamma _{n}\left( p_{0}\right) \right\} _{n\geq
1}$ is decreasing (increasing). This in combination with the facts that%
\begin{eqnarray*}
\gamma _{1}\left( p_{0}\right)  &=&\frac{\left( 1-a\right) \left( a+c\right)
}{c+1}\left( c-2b\right) >\left( <\right) 0, \\
\lim_{n\rightarrow \infty }\frac{\gamma _{n}\left( p_{0}\right) }{n^{2}}
&=&p_{0}\delta \left( p_{0}\right) <\left( >\right) 0
\end{eqnarray*}%
indicates that there is an integer $n_{1}>1$ such that $\gamma _{n}\left(
p_{0}\right) >\left( <\right) 0$ for $1\leq n<n_{1}$ and $\gamma _{n}\left(
p_{0}\right) <\left( >\right) 0$ for $n>n_{1}$. This completes the proof.
\end{proof}

\begin{remark}
\label{R-gp=0}If $\gamma _{n}=\gamma _{n}\left( p\right) \equiv 0$ for all $%
n\geq 1$, then%
\begin{equation*}
\left\{
\begin{array}{l}
p\left( a-b+c-p\right) =0, \\
cp^{2}-\left( a-b+c+2ab\right) p+a\left( a+c\right) =0, \\
ab\left( p\left( c+1\right) -\left( c+a\right) \right) =0.%
\end{array}%
\right.
\end{equation*}%
Under the conditions that $a\neq 0$, $c>\max \left\{ -a,b\right\} $ and $%
p\in \left( 0,1\right) $, this system of equations has a unique solution%
\begin{equation*}
c=1-2a=2b\text{ \ and \ }p=p_{0}=\frac{a+c}{1+c}=\frac{1}{2}.
\end{equation*}%
The recurrence relation (\ref{xin-rr-0}) is reduced to%
\begin{equation*}
\xi _{n+1}=\alpha _{n}\xi _{n}-\beta _{n}\xi _{n-1}
\end{equation*}%
for all $n\geq 1$, which together with $\xi _{0}=\xi _{1}=0$ yields $\xi
_{n}=0$ for all $n\geq 1$. Then $g_{p}\left( x\right) \equiv 0$, so $%
f_{p}^{\prime }\left( x\right) \equiv 0$ for $x\in \left( 0,1\right) $ by (%
\ref{dfp}). It follows that $f_{p}\left( x\right) \equiv 1$ for $x\in \left(
0,1\right) $, that is,%
\begin{equation*}
\left[ F\left( -a,\frac{1}{2}-a;1-2a;x\right) \right] ^{1/a}=\left( \frac{1+%
\sqrt{1-x}}{2}\right) ^{2}.
\end{equation*}%
This identity is a special case of the classical relation
\begin{equation*}
F\left( -a,b;2b;x\right) =\left( \frac{1+\sqrt{1-x}}{2}\right) ^{2a}F\left(
-a,\frac{1}{2}-a-b;\frac{1}{2}+b;\left( \frac{1-\sqrt{1-x}}{1+\sqrt{1-x}}%
\right) ^{2}\right)
\end{equation*}%
(see \cite[p. 132]{Andrews-SF-CUP-1999}).

Conversely, under the basic conditions that $a<1$, $b>0$, $c>\max \left\{
-a,b\right\} $, if $f_{p}\left( x\right) \equiv 1$ for all $x\in \left(
0,1\right) $, then $f_{p}^{\prime }\left( x\right) =0$, and then $%
g_{p}\left( x\right) =0$ for all $x\in \left( 0,1\right) $. This implies $%
\xi _{1}=\xi _{2}=\xi _{3}=0$, which yields $p=p_{0}$, $c=2b$, $2a+2b=1$.
That is to say, $f_{p}\left( x\right) \equiv 1$ for all $x\in \left(
0,1\right) $ if and only if $p=1/2$, $c=2b=1-2a$.
\end{remark}

\subsection{The signs of $g_{p}\left( 1\right) $ and $g_{p_{0}}\left(
1\right) $}

\begin{lemma}
\label{L-gp(1)}Let $a<1$, $b>0$, $c>\max \left\{ -a,b\right\} $ and let $%
g_{p}\left( x\right) $ be defined by (\ref{gp-a1}) with $p\in \left(
0,1\right) $. Then%
\begin{equation*}
\begin{array}{rll}
\text{(i)} & g_{p}\left( 1\right) =\infty & \text{if }p<c-b+a, \\
\text{(ii)} & g_{p}\left( 1\right) =-\infty & \text{if }p>c-b+a.%
\end{array}%
\end{equation*}
\end{lemma}

\begin{proof}
(i) Since $p\in \left( 0,1\right) $, we distinguish two cases to show that $%
g_{p}\left( 1\right) =\infty $ if $p<c-b+a$.

Case 1.1: $0<p<1\leq c-b+a$. By the first and second formulas of (\ref%
{F-near1}), we have%
\begin{equation*}
g_{p}\left( x\right) =\frac{F\left( 1-a,b;c+1;x\right) -\left( 1-x\right)
^{1-p}F\left( 1-a,b+1;c+1;x\right) }{\left( 1-x\right) ^{1-p}}\rightarrow
\infty
\end{equation*}%
as $x\rightarrow 1^{-}$.

Case 1.2: $0<p<c-b+a<1$. By the third formula of (\ref{F-near1}), $%
g_{p}\left( x\right) $ can be written as%
\begin{equation}
g_{p}\left( x\right) =\left( 1-x\right) ^{p-1}F\left( 1-a,b;c+1;x\right)
-\left( 1-x\right) ^{c-b+a-1}F\left( c+a,c-b;c+1;x\right)  \label{gp-a2}
\end{equation}%
\begin{equation*}
=\left( 1-x\right) ^{p-1}\left[ F\left( 1-a,b;c+1;x\right) -\left(
1-x\right) ^{c-b+a-p}F\left( c+a,c-b;c+1;x\right) \right] \rightarrow \infty
\end{equation*}%
as $x\rightarrow 1^{-}$.

(ii) Similarly, the proof can be divided into two cases.

Case 2.1: $0\leq c-b+a<p<1$. From (\ref{gp-a2}) we have%
\begin{equation*}
g_{p}\left( x\right) =\frac{\left( 1-x\right) ^{p-\left( c-b+a\right)
}F\left( 1-a,b;c+1;x\right) -F\left( c+a,c-b;c+1;x\right) }{\left(
1-x\right) ^{1-\left( c-b+a\right) }}\rightarrow -\infty
\end{equation*}%
as $x\rightarrow 1^{-}$.

Case 2.2: $c-b+a<0<p<1$. By the third formula of (\ref{F-near1}), the
relation (\ref{gp-a2}) can be changed to%
\begin{eqnarray*}
g_{p}\left( x\right) &=&\left( 1-x\right) ^{p-1+c+a-b}F\left(
c+a,c+1-b;c+1;x\right) \\
&&-\left( 1-x\right) ^{c-b+a-1}F\left( c+a,c-b;c+1;x\right) \\
&=&\frac{\left( 1-x\right) ^{p}F\left( c+a,c+1-b;c+1;x\right) -F\left(
c+a,c-b;c+1;x\right) }{\left( 1-x\right) ^{1-\left( c-b+a\right) }}%
\rightarrow -\infty
\end{eqnarray*}%
as $x\rightarrow 1^{-}$. This completes the proof.
\end{proof}

\begin{lemma}
\label{L-gp0(1)}Let $a<1$, $b>0$, $c>\max \left\{ -a,b\right\} $, and let $%
g_{p_{0}}\left( x\right) $ be defined by (\ref{gp-a1}), where $p_{0}$ is
defined by (\ref{p0}). Then%
\begin{equation*}
\begin{array}{rll}
\text{(i)} & g_{p_{0}}\left( x\right) =0 & \text{if }1/2=p_{0}=c-b+a, \\
\text{(ii)} & g_{p_{0}}\left( 1\right) =\infty & \text{if }0<p_{0}=c-b+a<1/2,
\\
\text{(iii)} & g_{p_{0}}\left( 1\right) =-\infty & \text{if }%
1/2<p_{0}=c-b+a<1.%
\end{array}%
\end{equation*}
\end{lemma}

\begin{proof}
Let%
\begin{equation*}
h_{p}\left( x\right) =\left( 1-x\right) ^{1-p}g_{p}\left( x\right) .
\end{equation*}%
If $p=c-b+a\in \left( 0,1\right) $, then by the representation \ref{gp-a2}
we have%
\begin{equation*}
h_{p}\left( x\right) =\left( 1-x\right) ^{p-1}g_{p}\left( x\right) =F\left(
1-a,b;c+1;x\right) -F\left( c+a,c-b;c+1;x\right) .
\end{equation*}
(i) If $p_{0}=a-b+c=1/2$, then $c=1-2a$ and $b=1/2-a$, and then,%
\begin{equation*}
h_{p}\left( x\right) =F\left( 1-a,1/2-a;c+1;x\right) -F\left(
1-a,1/2-a;c+1;x\right) \equiv 0\text{,}
\end{equation*}%
which implies that $g_{p_{0}}\left( x\right) \equiv 0$.

(ii) If $p_{0}=a-b+c\in \left( 0,1\right) $ with $p_{0}\neq 1/2$, an
application of the first formula of (\ref{F-near1}) yields%
\begin{equation*}
\left[ F\left( 1-a,b;c+1;x\right) -\left( 1-x\right) ^{c-b+a-p}F\left(
c+a,c-b;c+1;x\right) \right]
\end{equation*}%
\begin{equation*}
h_{p_{0}}\left( 1\right) =\dfrac{\Gamma \left( c+1\right) \Gamma \left(
c-b+a\right) }{\Gamma \left( c+a\right) \Gamma \left( c-b+1\right) }-\dfrac{%
\Gamma \left( c+1\right) \Gamma \left( 1-a+b-c\right) }{\Gamma \left(
1-a\right) \Gamma \left( 1+b\right) }.
\end{equation*}%
Solving the equations $p_{0}=\left( a+c\right) /\left( 1+c\right) =c-b+a$
for $a$ and $c$ gives%
\begin{equation*}
a=p_{0}\left( 1+c\right) -c\text{ \ and \ }b=cp_{0}.
\end{equation*}%
Then $h_{p_{0}}\left( 1\right) $ can be expressed as%
\begin{equation*}
h_{p_{0}}\left( 1\right) =\dfrac{\Gamma \left( c+1\right) \Gamma \left(
p_{0}\right) }{\Gamma \left( p_{0}+cp_{0}\right) \Gamma \left(
1+c-cp_{0}\right) }-\dfrac{\Gamma \left( c+1\right) \Gamma \left(
1-p_{0}\right) }{\Gamma \left[ \left( 1-p_{0}\right) \left( c+1\right) %
\right] \Gamma \left( 1+cp_{0}\right) }.
\end{equation*}%
Let%
\begin{equation*}
y\left( c\right) =\ln \dfrac{\Gamma \left( c+1\right) \Gamma \left(
p_{0}\right) }{\Gamma \left( p_{0}+cp_{0}\right) \Gamma \left(
1+c-cp_{0}\right) }-\ln \dfrac{\Gamma \left( c+1\right) \Gamma \left(
1-p_{0}\right) }{\Gamma \left[ \left( 1-p_{0}\right) \left( c+1\right) %
\right] \Gamma \left( 1+cp_{0}\right) }.
\end{equation*}%
Due to $y\left( 0\right) =0$, to prove parts (ii) and (iii) of this lemma,
it suffices to prove that $y^{\prime }\left( c\right) >0$ if $p_{0}\in
\left( 0,1/2\right) $ and $y^{\prime }\left( c\right) <0$ if $p_{0}\in
\left( 1/2,1\right) $.

Differentiation leads to
\begin{equation*}
y^{\prime }\left( c\right) =\left( 1-p_{0}\right) \left[ \psi \left(
1+c-p_{0}-cp_{0}\right) -\psi \left( 1+c-cp_{0}\right) \right] +p_{0}\left[
\psi \left( 1+cp_{0}\right) -\psi \left( p_{0}+cp_{0}\right) \right] ,
\end{equation*}%
which can be written as%
\begin{eqnarray*}
\frac{y^{\prime }\left( c\right) }{p_{0}\left( 1-p_{0}\right) } &=&\frac{%
\psi \left( 1+cp_{0}\right) -\psi \left( p_{0}+cp_{0}\right) }{1-p_{0}}-%
\frac{\psi \left( 1+c-cp_{0}\right) -\psi \left( 1+c-p_{0}-cp_{0}\right) }{%
p_{0}} \\
&:&=\frac{\psi \left( c_{1}\right) -\psi \left( d_{1}\right) }{c_{1}-d_{1}}-%
\frac{\psi \left( c_{2}\right) -\psi \left( d_{2}\right) }{c_{2}-d_{2}}.
\end{eqnarray*}%
Since $\psi ^{\prime \prime }\left( x\right) <0$ for $x>0$ and
\begin{eqnarray*}
c_{1}-c_{2} &=&\left( 1+cp_{0}\right) -\left( 1+c-cp_{0}\right) =c\left(
2p_{0}-1\right) <\left( >\right) 0, \\
d_{1}-d_{2} &=&\left( p_{0}+cp_{0}\right) -\left( 1+c-p_{0}-cp_{0}\right)
=\left( c+1\right) \left( 2p_{0}-1\right) <\left( >\right) 0,
\end{eqnarray*}%
if $p_{0}<\left( >\right) 1/2$, using the property of concave functions
gives $y^{\prime }\left( c\right) >0$ if $p_{0}\in \left( 0,1/2\right) $ and
$y^{\prime }\left( c\right) <0$ if $p_{0}\in \left( 1/2,1\right) $, thereby
completing the proof.
\end{proof}

\begin{remark}
\label{R-dp=0}Note that $p_{0}\leq \left( \geq \right) 1/2$ is equivalent to
$c\leq \left( \geq \right) 1-2a$. This together with the relations $%
p_{0}=c-b+a$ and (\ref{dp0-c-2b-1-2a}), that is,%
\begin{equation}
0=a-b+c-p_{0}=\frac{a+c}{c+1}\left( c-2b\right) +\frac{b}{c+1}\left[
c-\left( 1-2a\right) \right] .  \notag
\end{equation}%
means that $c>\left( <\right) 2b$. It then follows that%
\begin{eqnarray}
E_{+} &:&=\left\{ p_{0}=c-b+a\leq 1/2\right\} =\left\{ c\geq 2b\text{, }%
p_{0}=c-b+a\right\} ,  \label{E+} \\
E_{-} &:&=\left\{ p_{0}=c-b+a\geq 1/2\right\} =\left\{ c\leq 2b\text{, }%
p_{0}=c-b+a\right\} .  \label{E-}
\end{eqnarray}
\end{remark}

\section{Proof of Theorem \protect\ref{MT-m}}

\begin{proof}[Proof of Theorem \protect\ref{MT-m}]
A direct computation gives%
\begin{eqnarray*}
\xi _{0}\left( p\right) &=&0\text{, \ \ \ \ \ }\xi _{1}\left( p\right) =%
\frac{a+c}{c+1}-p, \\
\xi _{2}\left( p\right) &=&\frac{1}{2}p^{2}-\frac{1}{2}p\frac{2b+3c-2ab+3}{%
c+1}-\frac{\left( a+c\right) \left( a-b-c+ab-3\right) }{\left( c+1\right)
\left( c+2\right) },
\end{eqnarray*}%
and then, $\xi _{0}\left( p_{0}\right) =\xi _{1}\left( p_{0}\right) =0$ and%
\begin{equation*}
\xi _{2}\left( p_{0}\right) =\frac{1}{2}\frac{\left( 1-a\right) \left(
a+c\right) \left( c-2b\right) }{\left( c+1\right) ^{2}\left( c+2\right) }.
\end{equation*}%
For convenience, we always denote $H_{n}\left( p_{0}\right) $ by $H_{n}$ in
the proof below, where $H=\theta $, $\lambda $, $\gamma $, $\xi $, $\eta $.

\textbf{Part 1}: \textbf{Necessity}. If $f_{p_{0}}$ is increasing on $\left(
0,1\right) $, then $g_{p_{0}}\left( x\right) \geq 0$ for all $x\in \left(
0,1\right) $. Then the necessary condition follows from the inequalities%
\begin{equation}
\lim_{x\rightarrow 0}\frac{g_{p_{0}}\left( x\right) }{x^{2}}\geq 0\text{ \
and }g_{p_{0}}\left( 1^{-}\right) \geq 0.  \label{nci>0}
\end{equation}%
Since%
\begin{equation}
\lim_{x\rightarrow 0}\frac{g_{p_{0}}\left( x\right) }{x^{2}}=\xi _{2}\left(
p_{0}\right) =\frac{1}{2}\frac{\left( 1-a\right) \left( a+c\right) \left(
c-2b\right) }{\left( c+1\right) ^{2}\left( c+2\right) },  \label{nL1}
\end{equation}%
the first inequality of (\ref{nci>0}) implies that $c\geq 2b$. By Lemmas \ref%
{L-gp(1)} and \ref{L-gp0(1)}, the second inequality of (\ref{nci>0}) implies
that $a-b+c>p_{0}$ and $a-b+c=p_{0}\leq 1/2$. Therefor, the required
necessary conditions are%
\begin{equation*}
\left( a,b,c\right) \in \left\{ c\geq 2b\right\} \cap \left\{
a-b+c>p_{0}\right\} \cup \left\{ a-b+c=p_{0}\leq 1/2\right\} .
\end{equation*}%
According to Remark \ref{R-dp=0}, we see that%
\begin{equation*}
E_{+}=\left\{ a-b+c=p_{0}\leq 1/2\right\} =\left\{ c\geq 2b\text{, }%
p_{0}=c-b+a\right\} .
\end{equation*}%
We thus obtain the necessary conditions that are%
\begin{equation*}
\left( a,b,c\right) \in \left\{ c\geq 2b\text{, }a-b+c\geq p_{0}\right\} .
\end{equation*}

\textbf{Sufficiency}. Suppose that $c\geq 2b$ and $a-b+c\geq p_{0}$, that
is, $\left( a,b,c\right) \in \left\{ c\geq 2b,\delta \left( p_{0}\right)
\geq 0\right\} $. Then by Lemma \ref{L-sgngn-p0}, $\gamma _{n}\geq 0$ for
all $n\geq 1$. As shown in Lemma \ref{L-xin-rr-a}, $\xi _{n}$ satisfies the
recurrence relation%
\begin{equation}
\eta _{n+1}-\lambda _{n}\eta _{n}=\tau _{n}\gamma _{n}u_{n-1}>0
\label{xin+1-n-n-1-e-1}
\end{equation}%
for $n\geq 2$, where%
\begin{equation}
\eta _{n}=\xi _{n}-\theta _{n-1}\xi _{n-1},  \label{etan}
\end{equation}%
and $\theta _{n},\lambda _{n}>0$ for all $n\geq 1$. Since $\xi _{1}=0$, \ $%
\xi _{2}\geq 0$, so $\eta _{2}=\xi _{2}-\theta _{1}\xi _{1}=\xi _{2}\geq 0$,
by the relation (\ref{xin+1-n-n-1-e-1}) we have%
\begin{equation*}
\eta _{3}-\lambda _{2}\eta _{2}=\tau _{2}\gamma _{2}u_{1}\geq 0,
\end{equation*}%
which implies $\eta _{3}\geq \lambda _{2}\eta _{2}\geq 0$, so $\xi _{3}\geq
\theta _{2}\xi _{2}\geq 0$. Suppose that $\xi _{n}\geq 0$ for $2\leq n\leq m$%
. Using the relation (\ref{xin+1-n-n-1-e-1}) again, we obtain $\eta
_{n+1}-\lambda _{n}\eta _{n}\geq 0$ for $2\leq n\leq m$, which yields%
\begin{equation*}
\eta _{n+1}\geq \lambda _{n}\eta _{n}\geq \left( \lambda _{n}\lambda
_{n-1}\right) \eta _{n-1}\geq \cdot \cdot \cdot \geq \left( \lambda
_{n}\lambda _{n-1}\cdot \cdot \cdot \lambda _{2}\right) \eta _{2}\geq 0.
\end{equation*}%
It then follows that $\eta _{n+1}=\xi _{n+1}-\theta _{n}\xi _{n}\geq 0$, and
then $\xi _{n+1}\geq \theta _{n}\xi _{n}\geq 0$. By induction we arrive at $%
\xi _{n}\geq 0$ for all $n\in \mathbb{N}$, which implies $g_{p_{0}}\left(
x\right) =\sum_{n=0}^{\infty }\xi _{n}\left( p_{0}\right) x^{n}\geq 0$ for $%
x\in \left( 0,1\right) $. It is deduced from (\ref{dfp}) that the function $%
f_{p_{0}}$ is increasing on $\left( 0,1\right) $, and the double inequality (%
\ref{Pm<Hm<CPm-p0}) therefore holds for $x\in \left( 0,1\right) $.

\textbf{Part 2}: If $c\geq 2b$ and $a-b+c<p_{0}$, that is, $\left(
a,b,c\right) \in \left\{ c\geq 2b,\delta \left( p_{0}\right) <0\right\} $,
then by Lemma \ref{L-sgngn-p0}, there is an $n_{1}\geq 1$ such that $\gamma
_{n}\geq 0$ for $1\leq n\leq n_{1}$ and $\gamma _{n}<0$ for $n\geq n_{1}+1$.
By the relation (\ref{xin+1-n-n-1-e-1}), we have%
\begin{equation}
\eta _{n+1}-\lambda _{n}\eta _{n}=\tau _{n}\gamma _{n}u_{n-1}\left\{
\begin{array}{cc}
\geq 0 & \text{ if }1\leq n\leq n_{1}, \\
<0 & \text{if }n\geq n_{1}+1.%
\end{array}%
\right.   \label{etan-rr}
\end{equation}

\textbf{When }$1\leq n\leq n_{1}$. Similarly, $\xi _{2}-\theta _{1}\xi
_{1}=\xi _{2}\geq 0$, and by induction, $\xi _{n}\geq 0$ for $1\leq n\leq
n_{1}+1$.

\textbf{When }$n\geq n_{1}+1$. If there is no negative term in the
subsequence $\left\{ \xi _{n}\right\} _{n\geq n_{1}+2}$, then $\xi _{n}\geq
0 $ for all $n\geq 1$, which implies $\left\{ \xi _{n}\right\} _{n\geq 0}$
is a nonnegative sequence. This implies that the function $g_{p_{0}}\left(
x\right) $ is increasing on $\left( 0,1\right) $, which yields a
contradiction with $g_{p_{0}}\left( 1\right) <0$ by Lemma \ref{L-gp(1)}
(ii). Hence, there are some negative terms in the sequence $\left\{ \xi
_{n}\right\} _{n\geq n_{1}+2}$. Assume that the first negative term is $\xi
_{n_{0}}$ ($n_{0}\geq n_{1}+2$), then we can prove that all terms after that
are negative. In fact, since $\xi _{n_{0}-1}\geq 0$, $\xi _{n_{0}}<0$, by
Lemma \ref{L-xin-rr}, the recurrence relation (\ref{xin-rr-2}) implies that%
\begin{equation*}
\eta _{n_{0}}=\xi _{n_{0}+1}-\theta _{n_{0}}\xi _{n_{0}}=\lambda
_{n_{0}}\left( \xi _{n_{0}}+\theta _{n_{0}-1}\left( -\xi _{n_{0}-1}\right)
\right) +\tau _{n_{0}}\gamma _{n_{0}}u_{n_{0}-1}<0,
\end{equation*}%
which yields $\eta _{n_{0}}<0$, and so $\xi _{n_{0}+1}<\theta _{n_{0}}\xi
_{n_{0}}<0$.

Suppose that $\xi _{n}<0$ for $n_{0}+1\leq n\leq m$. Using the relation (\ref%
{etan-rr}) for $n\geq n_{1}+1$ yields $\eta _{n+1}-\lambda _{n}\eta
_{n}=\tau _{n}\gamma _{n}u_{n-1}<0$ for $n\geq n_{0}+1$. It then follows
that
\begin{equation*}
\eta _{n+1}<\lambda _{n}\eta _{n}<\left( \lambda _{n}\lambda _{n-1}\right)
\eta _{n-1}<\cdot \cdot \cdot <\left( \lambda _{n}\lambda _{n-1}\cdot \cdot
\cdot \lambda _{n_{0}}\right) \eta _{n_{0}}<0,
\end{equation*}%
which, by induction hypothesis, leads to $\xi _{n+1}<\theta _{n}\xi _{n}<0$.
We thus deduce that $\xi _{n}<0$ for $n\geq n_{0}$ by induction. So, $%
\left\{ \xi _{n}\right\} _{n\geq 0}$ is a PN-type sequence, and $%
g_{p_{0}}\left( x\right) =\sum_{n=0}^{\infty }\xi _{n}x^{n}$ is undoubtedly
a PN-type power series. By Proposition \ref{P-sgnS}, there is an $x_{0}\in
\left( 0,1\right) $ such that $g_{p_{0}}\left( x\right) >0$ for $x\in \left(
0,x_{0}\right) $ and $g_{p_{0}}\left( x\right) <0$ for $x\in \left(
x_{0},1\right) $. This, by (\ref{dfp}), shows that function $f_{p_{0}}$ is
increasing on $\left( 0,x_{0}\right) $ and decreasing on $\left(
x_{0},1\right) $. Consequently, the inequality (\ref{Pm<Hm-p0}) holds for $%
x\in \left( 0,1\right) $ if $a-b+c>0$.

\textbf{Part 3}: \textbf{Necessity}. If $f_{p_{0}}$ is decreasing on $\left(
0,1\right) $, then $g_{p_{0}}\left( x\right) \leq 0$ for all $x\in \left(
0,1\right) $. Then the necessary condition follows from the inequalities%
\begin{equation}
\lim_{x\rightarrow 0}\frac{g_{p_{0}}\left( x\right) }{x^{2}}\leq 0\text{ \
and }g_{p_{0}}\left( 1^{-}\right) \leq 0.  \label{nci<0}
\end{equation}%
It is seen from (\ref{nL1}) that the first inequality of (\ref{nci<0})
implies that $c\leq 2b$. By Lemmas \ref{L-gp(1)} and \ref{L-gp0(1)}, the
second inequality of (\ref{nci<0}) implies that $a-b+c<p_{0}$ and $%
a-b+c=p_{0}\geq 1/2$. Therefor, the required necessary conditions are%
\begin{equation*}
\left( a,b,c\right) \in \left\{ c\leq 2b\right\} \cap \left\{
a-b+c<p_{0}\right\} \cup \left\{ a-b+c=p_{0}\geq 1/2\right\} ,
\end{equation*}%
which, by (\ref{E-}) shown in Remark \ref{R-dp=0}, can be simplified to%
\begin{equation*}
\left( a,b,c\right) \in \left\{ c\leq 2b\text{, }a-b+c<p_{0}\right\} .
\end{equation*}%
This proves the necessity.

\textbf{Sufficiency}. Suppose that $c\leq 2b$ and $a-b+c\leq p_{0}$, that
is, $\left( a,b,c\right) \in \left\{ c\leq 2b,\delta \left( p_{0}\right)
\leq 0\right\} $. Applying the same technique as the sufficiency of Part 1,
we easily prove that $\xi _{n}\leq 0$ for all $n\in \mathbb{N}$, which
implies $g_{p_{0}}\left( x\right) =\sum_{n=0}^{\infty }\xi _{n}\left(
p_{0}\right) x^{n}\leq 0$ for $x\in \left( 0,1\right) $. It is deduced from (%
\ref{dfp}) that the function $f_{p_{0}}$ is decreasing on $\left( 0,1\right)
$.

\textbf{Part 4}: If $c\leq 2b$ and $a-b+c>p_{0}$, that is, $\left(
a,b,c\right) \in \left\{ c\leq 2b,\delta \left( p_{0}\right) >0\right\} $.
Then by Lemma \ref{L-sgngn-p0}, there is an $n_{1}\geq 1$ such that $\gamma
_{n}\leq 0$ for $1\leq n\leq n_{1}$ and $\gamma _{n}>0$ for $n\geq n_{1}+1$.
By the relation (\ref{xin+1-n-n-1-e-1}), we have%
\begin{equation*}
\eta _{n+1}-\lambda _{n}\eta _{n}=\tau _{n}\gamma _{n}u_{n-1}\left\{
\begin{array}{cc}
\leq 0 & \text{ if }1\leq n\leq n_{1}, \\
>0 & \text{if }n\geq n_{1}+1.%
\end{array}%
\right.
\end{equation*}%
In a similar way as Part 2, we can prove that the sequence $\left\{ \xi
_{n}\right\} _{n\geq 0}$ is an NP-type sequence, and therefore, $%
g_{p_{0}}\left( x\right) =\sum_{n=0}^{\infty }\xi _{n}\left( p_{0}\right)
x^{n}$ is an NP-type power series. Note that $g_{p_{0}}\left( 1\right)
=\infty $ due to Lemma \ref{L-gp(1)} (i). By Proposition \ref{P-sgnS}, there
is an $x_{0}\in \left( 0,1\right) $ such that $g_{p_{0}}\left( x\right) <0$
for $x\in \left( 0,x_{0}\right) $ and $g_{p_{0}}\left( x\right) >0$ for $%
x\in \left( x_{0},1\right) $. This, by (\ref{dfp}), shows that function $%
f_{p_{0}}$ is decreasing on $\left( 0,x_{0}\right) $ and increasing on $%
\left( x_{0},1\right) $. Consequently, the inequality (\ref{Pm<Hm-p0}) holds
for $x\in \left( 0,1\right) $.

This completes the proof.
\end{proof}

\section{Proofs of Theorems \protect\ref{MT-1} and \protect\ref{MT-2}}

We are now in a position to prove Theorem \ref{MT-1}.

\begin{proof}[Proof of Theorem \protect\ref{MT-1}]
Let%
\begin{eqnarray*}
Nc_{1} &=&\lim\limits_{x\rightarrow 0}\dfrac{a^{-1}\ln F\left(
-a,b;c;x\right) -\ln M_{p_{0}}\left( 1-x,b/c\right) }{x^{3}}, \\
Nc_{2} &=&\lim\limits_{x\rightarrow 1}\frac{\left[ F\left( -a,b;c;x\right) %
\right] ^{1/a}}{M_{p_{0}}\left( 1-x,b/c\right) }.
\end{eqnarray*}%
Direct computations yield%
\begin{eqnarray*}
Nc_{1} &=&\dfrac{b\left( 1-a\right) \left( a+c\right) \left( c-b\right) }{%
6c^{2}\left( c+1\right) ^{2}\left( c+2\right) }\left( c-2b\right) , \\
Nc_{2} &=&\left( 1-\dfrac{b}{c}\right) ^{-1/p_{0}}\left[ \dfrac{\Gamma
\left( c\right) \Gamma \left( c-b+a\right) }{\Gamma \left( c+a\right) \Gamma
\left( c-b\right) }\right] ^{1/a}.
\end{eqnarray*}

(i) In the case of $a-b+c\geq p_{0}$. The necessary conditions for (\ref%
{Hm>Pm}) to hold follow from the inequality $Nc_{1}\geq 0$, which gives $%
c\geq 2b$.

Suppose that $c\geq 2b$. Using Theorem \ref{MT-m} (i) yields%
\begin{equation*}
M_{p_{0}}\left( 1-x,\frac{b}{c}\right) \leq \left[ F\left( -a,b;c;x\right) %
\right] ^{1/a},
\end{equation*}%
which proves the sufficiency.

(ii) In the case of $a-b+c<p_{0}$. The first necessary condition must be: $%
c-b+a>0$. If not, that is, $c-b\leq -a$, which, due to $c>b>0$, implies that
$a<0$, then by the third formulas of (\ref{F-near1}) it yields that $%
\lim\limits_{x\rightarrow 1}\left[ F\left( -a,b;c;x\right) \right] ^{1/a}=0$%
. This leads to $Nc_{2}=0$. Under this necessary condition, solving the
inequalities $Nc_{1}\geq 0$ gives $c\geq 2b$. Note that $Nc_{2}=C_{p_{0}}$,
where $C_{p_{0}}$ is defined by (\ref{Cp}). By Remark \ref{R-Cp0-1}, $%
Nc_{2}=C_{p_{0}}\geq 1$ if and only if $p_{0}\leq p_{0}^{\ast }$. This
proves the necessity.

Suppose that $c-b+a>0$, $c\geq 2b$ and $p_{0}\leq p_{0}^{\ast }$. Using
Theorem \ref{MT-m} (ii) we have%
\begin{equation*}
\min \left\{ 1,C_{p_{0}}\right\} M_{p_{0}}\left( 1-x,\frac{b}{c}\right) \leq %
\left[ F\left( -a,b;c;x\right) \right] ^{1/a}
\end{equation*}%
for $x\in \left( 0,1\right) $. By Remark \ref{R-Cp0-1}, we see that $\min
\left\{ 1,C_{p_{0}}\right\} =1$ if $p_{0}\leq p_{0}^{\ast }$, which proves
the sufficiency, and the proof is done.
\end{proof}

Now let us prove Theorem \ref{MT-2}.

\begin{proof}[Proof of Theorem \protect\ref{MT-2}]
Let $Nc_{i}$ ($i=1,2$) be as in the proof of Theorem \ref{MT-1}. The
necessary conditions follows from the inequalities%
\begin{equation*}
\left\{
\begin{array}{l}
Nc_{1}=\dfrac{b\left( 1-a\right) \left( a+c\right) \left( c-b\right) }{%
6c^{2}\left( c+1\right) ^{2}\left( c+2\right) }\left( c-2b\right) \leq 0, \\
Nc_{2}=\left( 1-\dfrac{b}{c}\right) ^{-1/p_{0}}\left[ \dfrac{\Gamma \left(
c\right) \Gamma \left( c-b+a\right) }{\Gamma \left( c+a\right) \Gamma \left(
c-b\right) }\right] ^{1/a}\leq 1.%
\end{array}%
\right.
\end{equation*}%
(i) In the case of $a-b+c\leq p_{0}$. The necessary conditions for the
inequality (\ref{Hm<Pm}) holds for $x\in \left( 0,1\right) $ follow from the
inequality $Nc_{1}\leq 0$, which gives $p\geq p_{0}$ and $c\leq 2b$.

Suppose that $c\leq 2b$. Then using Theorem \ref{MT-m} (iii) yields%
\begin{equation*}
\left[ F\left( -a,b;c;x\right) \right] ^{1/a}\leq M_{p_{0}}\left( 1-x,\frac{b%
}{c}\right) ,
\end{equation*}%
which proves the sufficiency.

(ii) In the case of $a-b+c>p_{0}$. The first necessary condition for the
inequality (\ref{Hm<Pm}) to hold for $x\in \left( 0,1\right) $ follows from
the inequality $Nc_{1}\leq 0$, which gives $c\leq 2b$. Since $%
Nc_{2}=C_{p_{0}}$, where $C_{p_{0}}$ is defined by (\ref{Cp}). By Remark \ref%
{R-Cp0-1}, $Nc_{2}=C_{p_{0}}\leq 1$ if and only if $p_{0}\geq p_{0}^{\ast }$%
. This proves the necessity.

Suppose that $c\leq 2b$ and $p_{0}\geq p_{0}^{\ast }$. Using Theorem \ref%
{MT-m} (iv) we have%
\begin{equation*}
\left[ F\left( -a,b;c;x\right) \right] ^{1/a}\leq \max \left\{
1,C_{p_{0}}\right\} M_{p_{0}}\left( 1-x,\frac{b}{c}\right)
\end{equation*}%
for $x\in \left( 0,1\right) $. It is seen by Remark \ref{R-Cp0-1} that $\max
\left\{ 1,C_{p_{0}}\right\} =1$ if $p_{0}\geq p_{0}^{\ast }$, which proves
the sufficiency, thereby completing the proof.
\end{proof}

\section{Conclusions}

In this paper, using the sign rule of NP or PN-type power series and a
recurrence relation, we investigated the monotonicity of the ratio of
hypergeometric mean and power mean and proved our most important Theorems 1
and 2, which can also be seen in the following table (\textquotedblleft $+$%
\textquotedblright\ and \textquotedblleft $-$\textquotedblright\ means
\textquotedblleft $\geq 0$\textquotedblright\ and \textquotedblleft $\leq 0$%
\textquotedblright , respectively).%
\begin{equation*}
\begin{tabular}{|l|l|l|l|l|l|}
\hline
& $c-2b$ & $a-b+c-p_{0}$ & $f_{p_{0}}$ & $p_{0}-p$ & $f_{p}$ \\ \hline
& $+$ & $+$ & $\nearrow $ & $+$ & $\nearrow $ \\ \hline
& $+$ & $-$ & $\nearrow \searrow $ &  &  \\ \hline
& $-$ & $-$ & $\searrow $ & $-$ & $\searrow $ \\ \hline
& $-$ & $+$ & $\searrow \nearrow $ &  &  \\ \hline
\end{tabular}%
\end{equation*}%
From the monotonicity pattern of $f_{p_{0}}$ on $\left( 0,1\right) $, we
deduced the necessary and sufficient conditions for the first Hm-Pm
inequalities in all possible cases, and provide the first type of power mean
bounds (sharp at $x=0$) for the hypergeometric means. These greatly improve
and perfectly complement Barnard and Richards' results.

Finally, it should be emphasize that our two tools, Propositions \ref{P-sgnS}
and \ref{P-PH-rr}, play a crucial role in dealing with the problem on the
first type of power mean bounds for the hypergeometric mean. In our
forthcoming papers, we will give an answer to Conjecture \ref{Cj-BRT} by
virtue of the two tools.


\begin{thebibliography}{99}
\bibitem{Carlson-PAMS-16-1965} B. C. Carlson, A hypergeometric mean value,
\emph{Proc. Amer. Math. Soc.} \textbf{16} (1965), no. 4, 759--766.

\bibitem{Carlson-JMAA07-1963} B. C. Carlson, Lauricella'sh hypergeometric
function $F_{D}$, \emph{J. Math. Anal. Appl.} \textbf{7} (1963), 452--470.

\bibitem{Carlson-PAMS-16-1966} B. C. Carlson, Some inequalities for
hypergeometric functions, \emph{Proc. Amer. Math. Soc.} \textbf{16} (1966),
32--39.

\bibitem{Borwein-PA-NY-1974} J. M. Borwein and P. B. Borwein, \emph{Pi and
the AGM--A Study in Analytic Number Theory and Computational Complexity},
John Wiley, New York, 1987.

\bibitem{Almkvist-AMM-95-1988} G. Almkvist and B. Berndt, Gauss, Landen,
Ramanujan, the arithmetic-geometric mean, ellipses, pi, and the Ladies
Diary, \emph{Amer. Math. Monthly} \textbf{95} (1988), 585--608.

\bibitem{Barnard-JMA-31-2000} R.W. Barnard, K. Pearce, and K.C. Richards, An
inequality involving the generalized hypergeometric function and the arc
length of an ellipse, \emph{SIAM J. Math. Anal.} \textbf{31} (2000),
693--699.

\bibitem{Barnard-JMA-32-2000} R.W. Barnard, K. Pearce, and K.C. Richards, A
monotonicity property involving $_{3}F_{2}$ and comparisons of the classical
approximations of elliptical arc length, \emph{SIAM J. Math. Anal.} \textbf{%
32} (2000), 403--419.

\bibitem{Alzer-JCAM-172-2004} H. Alzer and S.-L. Qiu, Monotonicity theorems
and inequalities for the complete elliptic integrals, \emph{J. Comput. Appl.
Math.} \textbf{172} (2004), no. 2, 289--312.

\bibitem{Chu-RM-61-2012} Y.-M. Chu, M.-K. Wang, Optimal Lehmer mean bounds
for the Toader mean, \emph{Results Math.} \textbf{61} (2012), 223--229.

\bibitem{Yang-JMAA-438-2016} Z.-H. Yang, Y.-M. Chu, and W. Zhang, Accurate
approximations for the complete elliptic integral of the second kind, \emph{%
J. Math. Anal. Appl.} \textbf{438} (2016), 875--888.

\bibitem{Yang-JIA-2017-106} Z.-H. Yang, W.-M. Qian, Y.-M. Chu and W. Zhang,
Monotonicity rule for the quotient of two functions and its application,
\emph{J. Inequal. Appl.} \textbf{2017} (2017): 106.
%DOI 10.1186/s13660-017-1383-2.

\bibitem{Yang-JMAA-462-2018} Z.-H. Yang, W.-M. Qian, Y.-M. Chu, W. Zhang, On
approximating the arithmetic-geometric mean and complete elliptic integral
of the first kind, \emph{J. Math. Anal. Appl.} \textbf{462} (2018),
1714--1726.%, https://doi.org/10.1016/j.jmaa.2018.03.005.

\bibitem{Yang-JMAA-467-2018} Z.-H. Yang, Sharp approximations for the
complete elliptic integrals of the second kind by one-parameter means, \emph{%
J. Math. Anal. Appl.} \textbf{467} (2018), 446--461, 446--461. Available
online at https://doi.org/10.1016/j.jmaa.2018.07.020

\bibitem{Yang-AMC-348-2019} Z.-H. Yang, Y.-M. Chu and W. Zhang, High
accuracy asymptotic bounds for the complete elliptic integral of the second
kind, \emph{Appl. Math. Comput.} \textbf{348} (2019), 552--564.
%https://doi.org/10.1016/j.amc.2018.12.025.

\bibitem{Barnard-CMFT-1-2001} R. W. Barnard and K. C. Richards, A note on
the hypergeometric mean value, \emph{Comput. Meth. Func. Theory} \textbf{1}
(2001), no. 1, 81--88.

\bibitem{Richards-JMAA-308-2005} K. C. Richards, Sharp power mean bounds for
the Gaussian hypergeometric function, \emph{J. Math. Anal. Appl.} \textbf{308%
} (2005), 303--313. %doi:10.1016/j.jmaa.2005.01.018

\bibitem{Barnard-JIPAM-8-2007} R. W. Barnard and K. C. Richards, On
inequalities for hypergeometric analogues of the arithmetic-geometric mean,
\emph{J. Inequal. Pure Appl. Math.}, \textbf{8} (2007), no. 3, Art. 65.

\bibitem{Barnard-JMI-4-2010} R. W. Barnard, K. C. Richards and H. C.
Tiedeman, A survey of some bounds for Gauss' hypergeometric function and
related to mean, \emph{J. Math. Inequal.} \textbf{4} (2010), no. 1, 45--52.
DOI: dx.doi.org/10.7153/jmi-04-06

\bibitem{Pales-JMAA-131-1988} Z. P\'{a}les, Inequalities for differences of
powers, \emph{J. Math. Anal. App.} \textbf{131} (1988), 271--281.

\bibitem{Ling-AMM-81-1974} T.-P. Lin, The power mean and the logarithmic
mean, \emph{Amer. Math. Monthly}, \textbf{81 }(1974), 879--883.

\bibitem{Yang-JIA-2008-149286} Z.-H. Yang, On the monotonicity and
log-convexity of a four-parameter homogeneous mean, \emph{J. Inequal. Appl.}
\textbf{2008} (2008), Article ID 149286, 12 pages. %doi:10.1155/2008/149286

\bibitem{Yang-JIA-2013-116} Z.-H. Yang, New sharp bounds for logarithmic
mean and identric mean, \emph{J. Inequal. Appl.} \textbf{2013} (2013): 116.
%doi:10.1186/1029-242X-2013-116

\bibitem{Pittinger-UBPEFSMF-680-1980} A. O. Pittinger, Inequalities between
arithmetic and logarithmic means, \emph{Univ. Beogard Publ. Elektr. Fak.
Ser. Mat. Fiz} \textbf{680} (1980), 15--18.

\bibitem{Yang-MIA-10-2007} Z.-H. Yang, On the log-convexity of two-parameter
homogeneous functions, \emph{Math. Inequal. Appl.} \textbf{10 }(2007), no.
3, 499--516. %DOI: 10.7153/mia-10-47.

\bibitem{Yang-RACSAM-115-2021-8} Z.-H. Yang, J.-F. Tian, Y.-R. Zhu, A sharp
lower bound for the complete elliptic integrals of the first kind, \emph{%
Rev. R. Acad. Cienc. Exactas F\'{\i}s. Nat. Ser. A Mat.} \textbf{115}
(2021): 8. https://doi.org/10.1007/s13398-020-00949-6

\bibitem{Yang-AADM-13-2019} Z.-H. Yang and J. Tian, Convexity and
monotonicity for elliptic integrals of the first kind and applications,
\emph{Appl. Anal. Discrete Math.} \textbf{13} (2019), 240--260.
%https://doi.org/10.2298/AADM171015001Y.

\bibitem{Yang-arXiv-1705.05703} Z.-H. Yang, J. Tian, Convexity and
monotonicity for the elliptic integrals of the first kind and applications,
arXiv:1705.05703 [math.CA]. https://doi.org/10.48550/arXiv.1705.05703

\bibitem{Belzunce-I:M-E-40-2007} F. Belzunce, E. Ortega, and J. M. Ruiz, On
non-monotonic ageing properties from the Laplace transform, with actuarial
applications, \emph{Insurance Math. Econom.} \textbf{40} (2007), 1--14.

\bibitem{Yang-JIA-157-2015} Z.-H. Yang, Y.-M. Chu and X.-H. Zhang, Necessary
and sufficient conditions for functions involving the psi function to be
completely monotonic, \emph{J. Inequal. Appl.} \textbf{2015} (2015): 157.
%DOI: 10.1186/s13660-015-0674-8.

\bibitem{Yang-JIA-299-2015} Z.-H. Yang and Y.-M. Chu, Inequalities for
certain means in two arguments, \emph{J. Inequal. Appl.} \textbf{2015}
(2015): 299. %doi:10.1186/s13660-015-0828-8.

\bibitem{Yang-AAA-2014-702718} Z.-H. Yang, Y.-M. Chu, and X.-J, Tao, A
double inequality for the trigamma function and its applications, \emph{%
Abstr. Appl. Anal. }\textbf{2014} (2014), Art. ID 702718, 9 pages.
%http://dx.doi.org/10.1155/2014/702718.

\bibitem{Yang-JIA-2017-210} Z.-H. Yang, W.-M. Qian, Y.-M. Chu, W. Zhang, On
rational bounds for the gamma function, \emph{J. Inequal. Appl.} \textbf{2017%
} (2017): 210. %DOI: 10.1186/s13660-017-1484-y.

\bibitem{Yang-JMI-12-2018} Z.-H. Yang and J. Tian, Monotonicity and sharp
inequalities related to gamma function, \emph{J. Math. Inequal.} \textbf{12}
(2018), no. 1, 1--22. %doi:10.7153/jmi-2018-12-01.

\bibitem{Yang-MIA-21-2018-2} Z.-H. Yang, W.-M. Qian, Y.-M. Chu, W. Zhang, On
approximating the error function, \emph{Math. Inequal. Appl.} \textbf{21}
(2019), no. 2, 469--479. %doi:10.7153/mia-2018-21-32

\bibitem{Yang-MIA-21-2018-4} Z.-H. Yang, W.-M. Qian and Y.-M. Chu,
Monotonicity properties and bounds involving the complete elliptic integrals
of the first kind, \emph{Math. Inequal. Appl.} \textbf{21} (2018), no. 4,
1185--1199. %doi:10.7153/mia-2018-21-82.

\bibitem{Yang-JMAA-470-2019} Z. Yang and J.-F. Tian, Monotonicity rules for
the ratio of two Laplace transforms with applications, \emph{J. Math. Anal.
Appl.} \textbf{470} (2019), 821--845.
%https://doi.org/10.1016/j.jmaa.2018.10.034.

\bibitem{Yang-arxiv-2204.04709} Z.-H. Yang, Recurrence relations of
coefficients involving hypergeometric function with an application, arXiv:
2204.04709 [math.CA]. https://doi.org/10.48550/arXiv.2204.04709

\bibitem{Yang-RACSAM-115-2021} Z.-H. Yang and J.-F. Tian, Sharp bounds for
the Toader mean in terms of arithmetic and geometric means, \emph{Rev. R.
Acad. Cienc. Exactas F\'{\i}s. Nat. Ser. A Mat. RACSAM} \textbf{115}, 99
(2021). Available online at https://doi.org/10.1007/s13398-021-01040-4.

\bibitem{Zhao-RACSAM-117-2023} T. H. Zhao, M. K. Wang, Sharp bounds for the
lemniscatic mean by the weighted Holder mean, \emph{Rev. R. Acad. Cienc.
Exactas F\'{\i}s. Nat. Ser. A Mat. RACSAM} \textbf{117}, 96 (2023).
https://doi.org/10.1007/s13398-023-01429-3

\bibitem{Abramowitz-HMFFGMT-1972} M. Abramowitz and I. A. Stegun (Eds),
\emph{Handbook of Mathematical Functions with Formulas, Graphs, and
Mathematical Tables}, National Bureau of Standards, Applied Mathematics
Series 55, 10th printing, Dover Publications, New York and Washington, 1972.

\bibitem{Andrews-SF-CUP-1999} G.E. Andrews, R. Askey, R. Roy, \emph{Special
Functions}, Cambridge Univ. Press, Cambridge, 1999.
\end{thebibliography}
\end{document}